%% file: rf-rosen.tex
\documentclass[twoside
]{amsart}%

  \usepackage{a4wide}
  \usepackage[latin1]{inputenc}
  \usepackage{amsmath}
  \usepackage{amssymb}
  \usepackage{amsthm}
  \usepackage{mathrsfs}
  \usepackage{bbm}
  \usepackage{pifont}
  \usepackage{color}

  \usepackage{graphicx}

  \usepackage[normalem]{ulem}
	
	\def\AA{{\ifmmode{\mathbbm{A}}\else{$\mathbbm{A}$}\fi}}
	\def\BB{{\ifmmode{\mathbbm{B}}\else{$\mathbbm{B}$}\fi}}
	\def\CC{{\ifmmode{\mathbbm{C}}\else{$\mathbbm{C}$}\fi}}
	\def\EE{{\ifmmode{\mathbbm{E}}\else{$\mathbbm{E}$}\fi}}
	\def\FF{{\ifmmode{\mathbbm{F}}\else{$\mathbbm{F}$}\fi}}
	\def\HH{{\ifmmode{\mathbbm{H}}\else{$\mathbbm{H}$}\fi}}
	\def\KK{{\ifmmode{\mathbbm{K}}\else{$\mathbbm{K}$}\fi}}
	\def\NN{{\ifmmode{\mathbbm{N}}\else{$\mathbbm{N}$}\fi}}
	\def\PP{{\ifmmode{\mathbbm{P}}\else{$\mathbbm{P}$}\fi}}
	\def\QQ{{\ifmmode{\mathbbm{Q}}\else{$\mathbbm{Q}$}\fi}}
	\def\RR{{\ifmmode{\mathbbm{R}}\else{$\mathbbm{R}$}\fi}}
	\def\TT{{\ifmmode{\mathbbm{T}}\else{$\mathbbm{T}$}\fi}}
	\def\UU{{\ifmmode{\mathbbm{U}}\else{$\mathbbm{U}$}\fi}}
	\def\ZZ{{\ifmmode{\mathbbm{Z}}\else{$\mathbbm{Z}$}\fi}}
	
	\def\A{{\ifmmode{\mathscr{A}}\else{$\mathscr{A}$}\fi}}
	\def\B{{\ifmmode{\mathscr{B}}\else{$\mathscr{B}$}\fi}}
	\def\C{{\ifmmode{\mathscr{C}}\else{$\mathscr{C}$}\fi}}
	\def\D{{\ifmmode{\mathscr{D}}\else{$\mathscr{D}$}\fi}}
	\def\E{{\ifmmode{\mathscr{E}}\else{$\mathscr{E}$}\fi}}
	\def\F{{\ifmmode{\mathscr{F}}\else{$\mathscr{F}$}\fi}}
	\def\G{{\ifmmode{\mathscr{G}}\else{$\mathscr{G}$}\fi}}
	\def\H{{\ifmmode{\mathscr{H}}\else{$\mathscr{H}$}\fi}}
	\def\I{{\ifmmode{\mathscr{I}}\else{$\mathscr{I}$}\fi}}
	\def\J{{\ifmmode{\mathscr{J}}\else{$\mathscr{J}$}\fi}}
	\def\K{{\ifmmode{\mathscr{K}}\else{$\mathscr{K}$}\fi}}
	\def\L{{\ifmmode{\mathscr{L}}\else{$\mathscr{L}$}\fi}}
	\def\M{{\ifmmode{\mathscr{M}}\else{$\mathscr{M}$}\fi}}
	\def\N{{\ifmmode{\mathscr{N}}\else{$\mathscr{N}$}\fi}}
	\def\O{{\ifmmode{\mathscr{O}}\else{$\mathscr{O}$}\fi}}
	\def\P{{\ifmmode{\mathscr{P}}\else{$\mathscr{P}$}\fi}}
	\def\Q{{\ifmmode{\mathscr{Q}}\else{$\mathscr{Q}$}\fi}}
	\def\R{{\ifmmode{\mathscr{R}}\else{$\mathscr{R}$}\fi}}
	\def\S{{\ifmmode{\mathscr{S}}\else{$\mathscr{S}$}\fi}}
	\def\T{{\ifmmode{\mathscr{T}}\else{$\mathscr{T}$}\fi}}
	\def\U{{\ifmmode{\mathscr{U}}\else{$\mathscr{U}$}\fi}}
	\def\V{{\ifmmode{\mathscr{V}}\else{$\mathscr{V}$}\fi}}
	\def\W{{\ifmmode{\mathscr{W}}\else{$\mathscr{W}$}\fi}}
	\def\X{{\ifmmode{\mathscr{X}}\else{$\mathscr{X}$}\fi}}
	\def\Y{{\ifmmode{\mathscr{Y}}\else{$\mathscr{Y}$}\fi}}
	\def\Z{{\ifmmode{\mathscr{Z}}\else{$\mathscr{Z}$}\fi}}

	\newtheoremstyle{slanted}
	{}
	{}
	{\slshape}
	{}
	{\bfseries}
	{.}
	{ }
	{}
	
	\theoremstyle{slanted}
	\newtheorem{theo}{Theorem}[section]
	\newtheorem{prop}[theo]{Proposition}
	
	\newtheorem{remark}[theo]{Remark}
	
	\newtheorem{lemma}[theo]{Lemma}
	
	\newtheorem{corollary}[theo]{Corollary}
	
	\def\ind#1{\mathbbmss{1}_{#1}}
	\def\egdef{:=}
	\def\Id{\mathop{\mbox{Id}}}

 	\newcommand{\tend}[2]{\xrightarrow[#1\to#2]{}}
	\newcommand{\red}{\mathop{\rm Red}}

	\def\ind#1{\mathbbmss{1}_{#1}}

\title{Almost-sure Growth Rate of Generalized Random Fibonacci sequences}

\author{\'Elise Janvresse, Beno\^it Rittaud, Thierry de la Rue}

\address{\'Elise Janvresse, Thierry de la Rue:
Laboratoire de Math\'ematiques Rapha\"el Salem, 
Universit\'e de Rouen, CNRS -- 
Avenue de l'Universit\'e -- 
F76801 Saint \'Etienne du Rouvray.}
\email{Elise.Janvresse@univ-rouen.fr\\Thierry.de-la-Rue@univ-rouen.fr}
\address{Beno\^it Rittaud: Laboratoire Analyse, G\'eom\'etrie et Applications, Universit\'e Paris 13 Institut Galil\'ee, CNRS -- 
99 avenue Jean-Baptiste Cl\'ement -- 
F93 430 Villetaneuse.}
\email{rittaud@math.univ-paris13.fr}
\begin{document}
\bibliographystyle{amsplain}

\begin{abstract}
We study the generalized random Fibonacci sequences defined by their first nonnegative terms and for $n\ge 1$, 
$F_{n+2} = \lambda F_{n+1} \pm F_{n}$ (linear case) and 
$\widetilde F_{n+2} = |\lambda \widetilde F_{n+1} \pm \widetilde F_{n}|$ (non-linear case), 
where each $\pm$ sign is independent and either $+$ with probability $p$ or $-$ with probability $1-p$ ($0<p\le 1$). 
Our main result is that, when $\lambda$ is of the form $\lambda_k = 2\cos (\pi/k)$ for some integer $k\ge 3$, the exponential growth of $F_n$ for $0<p\le 1$, and of $\widetilde F_{n}$ for $1/k < p\le 1$, is almost surely positive and given by
$$
\int_0^\infty \log x\, d\nu_{k, \rho} (x), 
$$
where $\rho$ is an explicit function of $p$ depending on the case we consider, taking values in $[0, 1]$, and $\nu_{k, \rho}$ is an explicit probability distribution on $\RR_+$ defined inductively on generalized Stern-Brocot intervals. 
We also provide an integral formula for $0<p\le 1$ in the easier case $\lambda\ge 2$.
Finally, we study the variations of the exponent as a function of $p$.
\end{abstract}

\maketitle
\section{Introduction}
Random Fibonacci sequences have been defined by Viswanath by $F_1=F_2=1$ and the random recurrence 
$F_{n+2}= F_{n+1} \pm F_{n} $, where the $\pm$ sign is given by tossing a balanced coin. 
In~\cite{viswanath2000}, he proved that 
$$
\sqrt[n]{|F_n|}\longrightarrow1.13198824\ldots \quad\mbox{a.s.}
$$
and the logarithm of the limit is given by an integral expression involving a measure defined on Stern-Brocot intervals.
Rittaud~\cite{rittaud2006} studied the exponential growth of $\EE(|F_n|)$: it is given by an explicit algebraic number of degree 3, which turns out to be strictly larger than the almost-sure exponential growth obtained by Viswanath. 
In~\cite{janvresse2007}, Viswanath's result has been generalized to the case of an unbalanced coin and to the so-called non-linear case $F_{n+2}= |F_{n+1} \pm F_{n}|$. Observe that this latter case reduces to the linear recurrence when the $\pm$ sign is given by tossing a balanced coin.

A further generalization consists in fixing two real numbers, $\lambda$ and $\beta$, and considering 
the recurrence relation $F_{n+2}=\lambda F_{n+1}\pm \beta F_{n}$ (or $F_{n+2}=\vert \lambda F_{n+1}\pm \beta F_{n}\vert$), 
where the $\pm$ sign is chosen by tossing a balanced (or unbalanced) coin. 
By considering the modified sequence
$G_{n}:=F_{n}/\beta^{n/2}$, 
which satisfies $G_{n+2}=\frac{\lambda}{\sqrt{\beta}} G_{n+1}\pm G_{n}$, we can always reduce to the case $\beta=1$.
The purpose of this article is thus to generalize the results presented in~\cite{janvresse2007} on the almost-sure exponential growth to random Fibonacci sequences with a multiplicative coefficient: $(F_n)_{n\ge 1}$ and $(\widetilde F_n)_{n\ge 1}$, defined inductively by their first two positive terms $F_1=\widetilde F_1 =a$, $F_2=\widetilde F_2=b$
and for all $n\ge 1$, 
\begin{equation}
\label{linear case}
F_{n+2} = \lambda F_{n+1} \pm F_{n} \qquad \mbox{(linear case)},
\end{equation}
\begin{equation}
\label{non-linear case}
\widetilde F_{n+2} = |\lambda \widetilde F_{n+1} \pm \widetilde F_{n}| \qquad \mbox{(non-linear case)},
\end{equation}
where each $\pm$ sign is independent and either $+$ with probability $p$ or $-$ with probability $1-p$ ($0<p\le 1$). 
We are not yet able to solve this problem in full generality. If $\lambda\ge2$, the linear and non-linear cases are essentially the same, and the study of the almost-sure growth rate can easily be handled (Theorem~\ref{th:case_2}).
The situation $\lambda<2$ is much more difficult. However, the method developed in~\cite{janvresse2007} can be extended in a surprisingly elegant way to a countable family of $\lambda$'s, namely when $\lambda$ is of the form $\lambda_k = 2\cos (\pi/k)$ for some integer $k\ge 3$. 
The simplest case $\lambda_3=1$ corresponds to classical random Fibonacci sequences studied in~\cite{janvresse2007}. The link made in \cite{janvresse2007} and \cite{rittaud2006} between random Fibonacci sequences and continued fraction expansion remains valid for $\lambda_{k}=2\cos(\pi/k)$ and corresponds to so-called Rosen continued fractions, a notion introduced by Rosen in~\cite{rosen1954}.
These values $\lambda_{k}$ are the only ones strictly smaller than $2$ for which the group (called {\em Hecke group}) of transformations of the hyperbolic half plane $\HH^2$ generated by the transformations $z\longmapsto -1/z$ and $z\longmapsto z+\lambda$ is discrete.

In the linear case, the random Fibonacci sequence is given by a product of random i.i.d. matrices, and the classical way to investigate the exponential growth is to apply Furstenberg's formula~\cite{furstenberg1963}. 
This is the method used by Viswanath, and the difficulty lies in the determination of Furstenberg's invariant measure. 
In the non-linear case, the involved matrices are no more i.i.d., and the standard theory does not apply.
Our argument is completely different and relies on some reduction process which will be developed in details in the linear case. 
Surprisingly, our method works easier in the non-linear case, for which we only outline the main steps.

\bigskip

Our main results are the following.
\begin{theo}
\label{MainTheorem}
Let $\lambda=\lambda_k=2\cos (\pi/k)$, for some integer $k\ge 3$.

For any $\rho\in[0, 1]$, there exists an explicit probability distribution $\nu_{k, \rho}$ on $\RR_+$ defined inductively on generalized Stern-Brocot intervals (see Section~\ref{nu_rho} and Figure~\ref{mesure}), which gives the exponential growth of random Fibonacci sequences:
\begin{itemize}
\item {\bf Linear case:} Fix ${F}_1>0$ and ${F}_2>0$.
For $p=0$, the sequence $(|F_n|)$ is periodic with period $k$.
For any $p\in ]0,1]$, 
$$ 
\dfrac{1}{n} \log |F_n| \tend{n}{\infty}\gamma_{p,\lambda_k} = \int_0^\infty \log x\, d\nu_{k, \rho} (x) >0 $$
almost-surely, where 
$$
 \rho \egdef \sqrt[k-1]{1-p_R}
$$ 
and $p_R$ is the unique positive solution of $$\left(1-\dfrac{px}{p+(1-p)x}\right)^{k-1} = 1-x.$$
\item {\bf Non-linear case:} 
For $p\in ]1/k, 1]$ and any choice of ${\tilde F}_1>0$ and ${\tilde F}_2>0$,
$$ 
\dfrac{1}{n} \log \widetilde F_n \tend{n}{\infty} \widetilde \gamma_{p,\lambda_k} = \int_0^\infty \log x\, d\nu_{k, \rho} (x) >0$$
almost-surely, where 
$$
 \rho \egdef  \sqrt[k-1]{1-p_R}
$$
and $p_R$ is, for $p<1$, the unique positive solution of $$\left(1-\frac{px}{(1-p)+px}\right)^{k-1} = 1-x.$$
(For $p=1$, $p_R=1$.)
\end{itemize}
\end{theo}

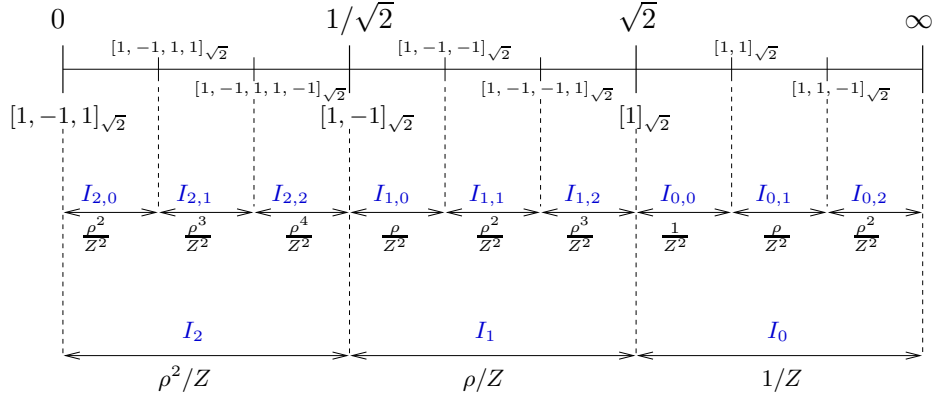
\begin{figure}[h]
	\begin{center}
	\input{mesure4.pstex_t}
	\end{center}
\caption{The measure $\nu_{k, \rho}$ on generalized Stern-Brocot intervals of rank 1 and 2 in the case $k=4$ ($\lambda_k=\sqrt{2}$). The normalizing constant $Z$ is given by $1+\rho+\rho^2$. The endpoints of the intervals are specified by their $\sqrt{2}$-continued fraction expansion.}
\label{mesure}
\end{figure} 

The behavior of $(\widetilde F_{n})$ when $p\le 1/k$ strongly depends on the choice of the initial values. This phenomenon was not perceived in~\cite{janvresse2007}, in which the initial values were set to $\widetilde F_{1}=\widetilde F_{2}=1$. However, we have the general result:

\begin{theo}
\label{Theorem2}
Let $\lambda=\lambda_k=2\cos (\pi/k)$, for some integer $k\ge 3$.
In the non-linear case, for $0\le p\le 1/k$, there exists almost-surely a bounded subsequence $(\widetilde F_{n_j})$ of $(\widetilde F_{n})$ with density $(1-kp)$. 
\end{theo}

The bounded subsequence in Theorem~\ref{Theorem2} satisfies $\widetilde F_{n_{j+1}} = |\lambda \widetilde F_{n_{j}}-\widetilde F_{n_{j-1}}|$ for any $j$, which corresponds to the non-linear case for $p=0$. 
We therefore concentrate on this case in Section~\ref{Sec:p=0} and provide necessary and sufficient conditions for $(\widetilde F_{n})$ to be ultimately periodic (see Proposition~\ref{p=0}). Moreover, we prove that $\widetilde F_{n}$ may decrease exponentially fast to $0$, but that the exponent depends on the ratio $\widetilde F_0/\widetilde F_1$. 

The critical value $1/k$ in the non-linear case is to be compared with the results obtained in the study of $\EE[\widetilde F_n]$ (see~\cite{janvresse2008}): it is proved that $\EE[\widetilde F_n]$ increases exponentially fast as soon as $p>(2-\lambda_k)/4$.

\bigskip
When $\lambda\ge 2$, the linear case and the non-linear case are essentially the same. The study of the exponential growth of the sequence $(F_n)$ is much simpler, and we obtain the following result.

\begin{theo}\label{th:case_2}
Let $\lambda\ge 2$ and $0<p\le 1$. 
For any choice of $F_1>0$ and $F_2>0$,
$$ 
\dfrac{1}{n} \log |F_n| \tend{n}{\infty} \gamma_{p,\lambda} = \int_0^\infty \log x\, d\mu_{p, \lambda} (x)>0\quad\mbox{a.s.}, $$
where $\mu_{p, \lambda}$ is an explicit probability measure supported on $\left[B, \lambda+\frac{1}{B}\right]$, with $B\egdef\dfrac{\lambda+\sqrt{\lambda^2-4}}{2}$ (see Section~\ref{Sec:lambda2} and Figure~\ref{mesure2}).
\end{theo}

\subsection*{Road map}

The detailed proof of Theorem~\ref{MainTheorem} in the linear case is given in Sections~\ref{reductionSection}-\ref{Section:positivity}: Section~\ref{reductionSection} explains the reduction process on which our method relies. In Section~\ref{Sec:SternBrocot}, we introduce the generalized Stern-Brocot intervals in connection with the expansion of real numbers in Rosen continued fractions, which enables us to study the reduced sequence associated to $(F_n)$. In Section~\ref{Sec:Coupling}, we come back to the original sequence $(F_n)$, and, using a coupling argument, we prove that its exponential growth is given by the integral formula. Then we prove the positivity of the integral in Section~\ref{Section:positivity}.

The proof for the non-linear case, $p>1/k$, works with the same arguments (in fact it is even easier), and the minor changes are given at the beginning of Section~\ref{Section:non-linear}. The end of this section is devoted to the proof of Theorem~\ref{Theorem2}. 

The proof of Theorem~\ref{th:case_2} (for $\lambda\ge2$) is given in Section~\ref{Sec:lambda2}.

In Section~\ref{Sec:croissance_p}, we study the variations of $\gamma_{p,\lambda}$ and $\widetilde\gamma_{p,\lambda}$ with~$p$.
Conjectures concerning variations with $\lambda$ are given in Section~\ref{Sec:variations_k}.

Connections with Embree-Trefethen's paper~\cite{embree1999}, who study a slight modification of our linear random Fibonacci sequences when $p=1/2$, are discussed in Section~\ref{Sec:Embree-Trefethen}.

\section{Reduction: The linear case}
\label{reductionSection}

The sequence $(F_n)_{n\ge1}$ can be coded by a sequence $(X_n)_{n\ge 3}$ of i.i.d. random variables taking values in the alphabet $\{R, L\}$ with probability $(p, 1-p)$. 
Each $R$ corresponds to choosing the $+$ sign and each $L$ corresponds to choosing the $-$ sign, so that both can be interpreted as the right multiplication of $(F_{n-1}, F_n)$ by one of the following matrices:
\begin{equation}
\label{matrices}
L \egdef 
\begin{pmatrix}
0 & -1\\
1 & \lambda
\end{pmatrix}
\qquad \mbox{and}\qquad 
R \egdef 
\begin{pmatrix}
0 & 1\\
1 & \lambda
\end{pmatrix}.
\end{equation}
According to the context, we will interpret any finite sequence of $R$'s and $L$'s as the corresponding product of matrices. 
Therefore, for all $n\ge3$,
$$(F_{n-1}, F_n) = (F_1, F_2)X_3\ldots X_n.$$

Our method relies on a reduction process of the sequence $(X_n)$ based on some relations satisfied by the matrices $R$ and $L$.
Recalling the definition of $\lambda=2\cos(\pi/k)$, we can write the matrix $L$ as the product $P^{-1}DP$, where 
$$ D\egdef\begin{pmatrix}
e^{i\pi/k} & 0          \\
0          & e^{-i\pi/k}
\end{pmatrix},
\quad
P\egdef\begin{pmatrix}
1   &  e^{i\pi/k}          \\
1   &  e^{-i\pi/k}
\end{pmatrix},
\quad\mbox{and }
P^{-1}=\dfrac{1}{2i\sin(\pi/k)}\begin{pmatrix}
-e^{-i\pi/k}   &  e^{i\pi/k}          \\
1              &  -1
\end{pmatrix}.
$$
As a consequence, we get that for any integer $j$,
\begin{equation}
 \label{PowersOfL}
 L^j =\dfrac{1}{\sin(\pi/k)}\begin{pmatrix}
-\sin\frac{(j-1)\pi}{k}  &  -\sin\frac{j\pi}{k}      \\
 \sin\frac{j\pi}{k}      &  \sin\frac{(j+1)\pi}{k}
\end{pmatrix},
\end{equation}
and
\begin{equation}
 \label{RtimesPowersOfL}
 RL^j =\dfrac{1}{\sin(\pi/k)}\begin{pmatrix}
 \sin\frac{j\pi}{k}      &  \sin\frac{(j+1)\pi}{k}      \\
 \sin\frac{(j+1)\pi}{k}  &  \sin\frac{(j+2)\pi}{k}
\end{pmatrix}.
\end{equation}
In particular, for $j=k-1$ we get the following relations satisfied by $R$ and $L$, on which is based our reduction process:
\begin{equation}
\label{reduction}
RL^{k-1} = 
\begin{pmatrix}
1 & 0\\
0 & -1
\end{pmatrix},
\quad
 RL^{k-1}R = -L \quad\mbox{and}\quad RL^{k-1}L=-R.
\end{equation}
Moreover, $L^k = -\Id$.

\medskip
We deduce from \eqref{reduction} that, in products of $R$'s and $L$'s, we can suppress all patterns $RL^{k-1}$ provided we flip the next letter. This will only affect the sign of the resulting matrix.

To formalize the reduction process, we associate to each finite sequence $x=x_3\ldots x_n\in\{R,L\}^{n-2}$ a (generally) shorter word $\red(x)=y_3\cdots y_{j}$ by the following induction. 
If $n=3$, $y_3 = x_3$. 
If $n>3$, $\red(x_3\ldots x_n)$ is deduced from $\red(x_3\ldots x_{n-1})$ in two steps. 

{\bf Step 1}: Add one letter ($R$ or $L$, see below) to the end of $\red(x_3\ldots x_{n-1})$. 

{\bf Step 2}: If the new word ends with the suffix $RL^{k-1}$, remove this suffix. 

The letter which is added in step 1 depends on what happened when constructing $\red(x_3\ldots x_{n-1})$: 
\begin{itemize}
\item If $\red(x_3\ldots x_{n-1})$ was simply obtained by appending one letter, we add $x_{n}$ to the end of $\red(x_3\ldots x_{n-1})$. 
\item Otherwise, we had removed the suffix $RL^{k-1}$ when constructing $\red(x_3\ldots x_{n-1})$; we then add $\overline{x_n}$ to the end of $\red(x_3\ldots x_{n-1})$, where $\overline{R}\egdef L$ and $\overline{L}\egdef R$. 
\end{itemize}
Example: Let $x=RLRLLLRLL$ and $k=4$. Then, the reduced sequence is given by $\red(x) = R$. 

\medskip
Observe that by construction, $\red(x)$ never contains the pattern $RL^{k-1}$. 
Let us introduce the \emph{reduced random Fibonacci sequence} $(F_n^r)$ defined by
$$(F_{n-1}^r,F_n^r) \egdef  (F_1,F_2) \red(X_3\ldots X_n).$$
Note that we have $ F_n=\pm F_n^r$ for all $n$.
From now on, we will therefore concentrate our study on the reduced sequence $\red(X_3\ldots X_n)$. We will denote its length by $j(n)$ and its last letter by $Y(n)$.

The proof of Lemma~2.1 in~\cite{janvresse2007} can be directly adapted to prove the following lemma.
\begin{lemma}
\label{Survival}
We denote by $|W|_R$ the number of $R$'s in the word $W$. We have
\begin{equation}
\label{RnR}
|\red(X_3\ldots X_n)|_R\tend{n}{\infty}+\infty\qquad\mbox{a.s.}
\end{equation}
In particular, the length $j(n)$ of $\red(X_3\ldots X_n)$ satisfies
$$
j(n) \tend{n}{\infty} +\infty.\qquad\mbox{a.s.}
$$
\end{lemma}

\subsection{Survival probability of an $R$}
\label{section:survival}
We say that the last letter of $\red(X_3\ldots X_n)$ \emph{survives} if, for all $m\ge n$, $j(m)\ge j(n)$. In other words, this letter survives if it is never removed during the subsequent steps of the reduction.
By construction, the survival of the last letter $Y(n)$ of $\red(X_3\ldots X_n)$ only depends on its own value and the future $X_{n+1},X_{n+2}\ldots$. Let
$$ 
p_R\egdef\PP\Bigl( Y(n)\mbox{ survives }\Big|Y(n)=R \mbox{ has been appended at time }n\Bigr).
$$
A consequence of Lemma~\ref{Survival} is that $p_R >0$. We now want to express $p_R$ as a function of $p$. 

Observe that $Y(n)=R$ survives if and only if, after the subsequent steps of the reduction, it is followed by $L^jR$ where $0\le j\le k-2$, and the latter $R$ survives. Recall that the probability of appending an $R$ after a deletion of the pattern $RL^{k-1}$ is $1-p$, whereas it is equal to $p$ if it does not follow a deletion. 
Assume that $Y(n)=R$  has been appended at time $n$. We want to compute the probability for this $R$ to survive and to be followed by $L^jR$ ($0\le j\le k-2$) after the reduction. This happens with probability
\begin{eqnarray*}
p_j  & \egdef & 
\PP\left( R \mbox{ be followed by }
\underbrace{\Bigl(\!\!\overbrace{[\mbox{\sout{$R\ldots$}}]}^{\substack{\ell\ge 0 \\ \mbox{\scriptsize\ deletions}}}\!\!\! L\Bigr)\ \ldots\ \Bigl([\mbox{\sout{$R\ldots$}}] \,L\Bigr)}_{j\mbox{\scriptsize\ times}}
\quad [\mbox{\sout{$R\ldots$}}]\,  \overbrace{R}^{\mbox{\scriptsize survives}}\right) \\
& = &\left( (1-p)+ p\sum_{\ell\ge1} (1-p_R)^\ell (1-p)^{\ell-1}p \right)^j p\sum_{\ell\ge0} (1-p_R)^\ell (1-p)^{\ell} p_R\\
& = &  \left(1-\dfrac{pp_R}{p+(1-p)p_R}\right)^j \dfrac{pp_R}{p+(1-p)p_R}.
\end{eqnarray*}
Writing $p_R = \sum_{j=0}^{k-2}p_j$, we get that $p_R$ is a solution of the equation
\begin{equation}
\label{survival-pr}
 g(x)=0,\quad\mbox{where }g(x)\egdef 1- \dfrac{px}{p+(1-p)x} - (1-x)^{1/(k-1)}.
\end{equation}
Observe that $g(0)=0$, and that $g$ is strictly convex. Therefore there exists at most one $x>0$ satisfying $g(x)=0$, and it follows that $p_R$ is the unique positive solution of~\eqref{survival-pr}. 

\subsection{Distribution law of surviving letters}
A consequence of Lemma~\ref{Survival} is that the sequence of surviving letters 
$$ (S_j)_{j\ge3} = \lim_{n\to\infty} \red(X_3\ldots X_n) $$
is well defined and can be written as the concatenation of a certain number $s\ge0$ of starting $L$'s, followed by infinitely many blocks: 
$$
S_1S_2 \ldots = L^s B_1B_2\ldots
$$
where $s\ge 0$ and, for all $\ell\ge1$, $B_\ell\in\{R, RL, \ldots, RL^{k-2}\}$.
This block decomposition will play a central role in our analysis. 

We deduce from Section~\ref{section:survival} the probability distribution of this sequence of blocks:
\begin{lemma}
 \label{law}
The blocks $(B_\ell)_{\ell\ge1}$ are i.i.d. with common distribution law $\PP_{\rho}$ defined as follows
\begin{equation}
 \label{defPrho}
\PP_{\rho}(B_1 = RL^j)\egdef \frac{\rho^j}{\sum_{m=0}^{k-2}\rho^m}\ ,\quad 0\le j\le k-2, 
\end{equation}
where $\rho \egdef 1-\dfrac{pp_R}{p+(1-p)p_R}$ and $p_R$ is the unique positive solution of~\eqref{survival-pr}.
\end{lemma}
In \cite{janvresse2007}, where the case $k=3$ was studied, we used the parameter $\alpha = 1/(1+\rho)$ instead of $\rho$.

Observe that $\rho=\left( (1-p)+ p\sum_{\ell\ge1} (1-p_R)^\ell (1-p)^{\ell-1}p \right)$ can be interpreted as the probability that the sequence of surviving letters starts with an $L$. Since an $R$ does not survive if it is followed by $k-1$ $L$'s, this explains why the probability $1-p_R$ that an $R$ does not survive is equal to $\rho^{k-1}$.

\begin{proof}
Observe that the event $E_n:=$``$Y(n)=R$ has been appended at time $n$ and survives'' is the intersection of the two events ``$Y(n)=R$ has been appended at time $n$'', which is measurable with respect to $\sigma(X_i,\ i\le n)$, and ``If $Y(n)=R$ has been appended at time $n$, then this $R$ survives'', which is measurable with respect to $\sigma(X_i,\ i>n)$.
It follows that, conditioned on $E_n$, $\sigma(X_i,\ i\le n)$ and $\sigma(X_i,\ i> n)$ remain independent. 
Thus the blocks in the sequence of surviving letters appear independently, and their distribution is given by 
$$\PP_{\rho}(B_1 = RL^j)=\frac{p_j}{p_R} = \frac{\rho^j}{\sum_{m=0}^{k-2}\rho^m}\ ,\quad 0\le j\le k-2.$$
\end{proof}

\section{Rosen continued fractions and generalized Stern-Brocot intervals}
\label{Sec:SternBrocot}

\subsection{The quotient Markov chain}
For $\ell\ge1$, let us denote by $n_\ell$ the time when the $\ell$-th surviving $R$ is appended, and set 
$$Q_\ell\egdef \dfrac{F_{n_{\ell+1}-1}^r}{F_{n_{\ell+1}-2}^r}, \quad\ell\ge0. $$
$Q_\ell$ is the quotient of the last two terms once the $\ell$-th definitive block of the reduced sequence has been written.
Observe that the right-product action of blocks $B\in\{R, RL, \ldots, RL^{k-2}\}$ acts on the quotient $F_n^r/F_{n-1}^r$ in the following way: For $0\le j\le k-2$, for any $(a,b)\in\RR^*\times\RR$, if we set $(a',b')\egdef (a,b)RL^j$, then
$$ \dfrac{b'}{a'} = f^j\circ f_0 \left(\dfrac{b}{a}\right), $$
where $f_0(q)\egdef\lambda+1/q$ and $f(q)\egdef\lambda -1/q$.
For short, we will denote by $f_j$ the function $f^j\circ f_0$.
Observe that $f_j$ is an homographic function associated to the matrix $RL^j$ in the following way: To the matrix $\begin{pmatrix}
\alpha  &  \beta      \\
\gamma  &  \delta
\end{pmatrix}$ corresponds the homographic function $q\mapsto\frac{\beta+\delta q}{\alpha+\gamma q}$.

It follows from Lemma~\ref{law} that $(Q_\ell)_{\ell\ge1}$ is a real-valued Markov chain with probability transitions 
$$
\PP\left( Q_{\ell+1} = f_j(q) | Q_{\ell} = q \right) = \frac{\rho^j}{\sum_{m=0}^{k-2}\rho^m}\ ,\quad 0\le j\le k-2.
$$

\subsection{Generalized Stern-Brocot intervals and the measure $\nu_{k, \rho}$}
\label{nu_rho}
Let us define subintervals of $\RR$: for $0\le j\le k-2$, set $I_j\egdef f_j([0, +\infty ])$. 
These intervals are of the form 
$$I_j =[b_{j+1}, b_j], \mbox{ where } b_0=+\infty, \ b_1=\lambda = f_0(+\infty)=f_1(0), \ b_{j+1}=f(b_j) = f_{j}(+\infty)=f_{j+1}(0).$$
Observe that $b_{k-1}=f_{k-1}(0)=0$ since $RL^{k-1}=\begin{pmatrix}1& 0\\0&-1\end{pmatrix}$.
Therefore, $(I_j)_{0\le j\le k-2}$ is a subdivision of $[0, +\infty ]$.

More generally, we set 
$$
I_{j_1, j_2, \dots, j_\ell}\egdef f_{j_1}\circ f_{j_2}\circ \cdots \circ f_{j_\ell} ([0, +\infty ]), \quad 
\forall (j_1, j_2, \dots, j_\ell)\in \{0, \dots, k-2\}^\ell.
$$
For any $\ell\ge1$, this gives a subdivision $\I(\ell)$ of $[0, +\infty ]$ since 
$$I_{j_1, j_2, \dots, j_{\ell-1}} = \bigcup_{j_\ell=0}^{k-2} I_{j_1, j_2, \dots, j_\ell}.$$
When $k=3$ ($\lambda =1$), this procedure provides subdivisions of $[0, +\infty ]$ into Stern-Brocot intervals.

\begin{lemma}
\label{generation}
The $\sigma$-algebra generated by $\I(\ell)$ increases to the Borel $\sigma$-algebra on $\RR_+$.
\end{lemma}
We postpone the proof of this lemma to the next section.

Observe that for any $q\in\RR_+$, $\PP(Q_{\ell}\in I_{j_1, j_2, \dots, j_{\ell}}| Q_0=q)=\frac{\rho^{j_1+\cdots+j_\ell}}{(\sum_0^{k-2}\rho^m)^\ell}$.
Therefore, the probability measure $\nu_{k, \rho}$ on $\RR_+$ defined by 
$$
\nu_{k, \rho}(I_{j_1, j_2, \dots, j_{\ell}}) \egdef \frac{\rho^{j_1+\cdots+j_\ell}}{(\sum_0^{k-2}\rho^m)^\ell}\ 
$$
is invariant for the Markov chain $(Q_\ell)$. 
The fact that $\nu_{k, \rho}$ is the unique invariant probability for this Markov chain comes from the following lemma.
\begin{lemma}
\label{L+}
 There exists almost surely $L_+\ge0$ such that for all $\ell\ge L_+$, $Q_{\ell}>0$.
\end{lemma}
 
\begin{proof}
 For any $q\in\RR\setminus\{0\}$, either $f_0(q)>0$, or $f_1(q)=\lambda-1/f_0(q)>0$. Hence, for any $\ell\ge0$,
$$\PP(Q_{\ell+1}>0|Q_\ell=q)\ge \frac{\rho}{\sum_0^{k-2}\rho^m}\,. $$
It follows that $\PP(\forall\ell\ge0,\ Q_\ell<0)=0$, and since $Q_\ell>0\Longrightarrow Q_{\ell+1}>0$, the lemma is proved.
\end{proof}

To a given finite sequence of blocks $(RL^{j_\ell}), \ldots, (RL^{j_1})$, we associate the generalized Stern-Brocot interval $I_{j_1, j_2, \dots, j_{\ell}}$. If we extend the sequence of blocks leftwards, we get smaller and smaller intervals.
Adding infinitely many blocks, we get in the limit a single point corresponding to the intersection of the intervals, which follows the law $\nu_{k, \rho}$.

\subsection{Link with Rosen continued fractions}
Recall (see~\cite{rosen1954}) that, since $1\le \lambda <2$, any real number $q$ can be written as
$$ q = a_0\lambda + \cfrac{1}{a_1\lambda + \cfrac{1}{\ddots+\cfrac{1}{a_n \lambda +_{\ddots}}}} $$
where $(a_n)_{n\ge0}$ is a finite or infinite sequence, with $a_n\in\ZZ\setminus\{0\}$ for $n\ge1$.
This expression will be denoted by $[a_0,\ldots, a_n,\ldots]_\lambda$. It is called a \emph{$\lambda$-Rosen continued fraction expansion} of $q$, and is not unique in general. When $\lambda=1$ (\textit{i.e.} for $k=3$), we recover generalized continued fraction expansion in which partial quotients are positive or negative integers. 

Observe that the function $f_j$ are easily expressed in terms of Rosen continued fraction expansion. The Rosen continued fraction expansion of $f_j(q)$ is the concatenation of $(j+1)$ alternated $\pm1$ with the expansion of $\pm q$ according to the parity of $j$:
\begin{equation}
\label{fj}
 f_j([a_0,\ldots, a_n,\ldots]_\lambda) = 
\begin{cases}
[\underbrace{1,-1,1,\ldots,1}_{(j+1) \mbox{ terms}},a_0,\ldots, a_n,\ldots]_\lambda & \mbox{ if $j$ is even}\\
[\underbrace{1,-1,1,\ldots,-1}_{(j+1) \mbox{ terms}},-a_0,\ldots, -a_n,\ldots]_\lambda & \mbox{ if $j$ is odd}.
\end{cases}
\end{equation}

For any $\ell\ge1$, let $\E(\ell)$ be the set of endpoints of the subdivision $\I(\ell)$.
The finite elements of $\E(1)$ can be written as
$$ b_j=f_j(0)=[\underbrace{1,-1,1,\ldots,\pm1}_{j \mbox{ terms}}]_\lambda \quad \forall\ 1\le j\le k-1.$$ In particular for $j=k-1$ we get a finite expansion of $b_{k-1}=0$. Moreover, by~\eqref{fj},
$$b_0=f_0(0)=\infty=[1,\underbrace{1,-1,1,\ldots,\pm1}_{k-1 \mbox{ terms}}]_\lambda.$$
Iterating~\eqref{fj}, we see that for all $\ell\ge1$, the elements of $\E(\ell)$ can be written as a finite $\lambda$-Rosen continued fraction with coefficients in $\{-1,1\}$.

\begin{prop}
\label{finiteCF}
The set $\bigcup_{\ell\ge1}\E_\ell$ of all endpoints of generalized Stern-Brocot intervals is the set of all nonnegative real numbers admitting a finite $\lambda$-Rosen continued fraction expansion.
\end{prop}

The proof uses the two following lemmas

\begin{lemma}
\label{inverse}
$$
f_j(x)=\dfrac{1}{f_{k-2-j}(1/x)}, \quad\forall\ 0\le j\le k-2.
$$
\end{lemma}

\begin{proof}
From \eqref{RtimesPowersOfL}, we get 
$$RL^{k-2}=\begin{pmatrix} \lambda & 1 \\ 1 & 0 \end{pmatrix}, \quad\mbox{hence}\quad f_{k-2}(x)=\frac{1}{\lambda+x}.$$ 
Therefore, $f_{k-2}(1/x)=1/f_{0}(x)$ and the statement is true for $j=0$. Assume now that the result is true for $j\ge0$. 
We have
$$
f_{j+1}(x)
= \lambda - \dfrac{1}{f_j(x)} = \lambda - f_{k-2-j}\left(\dfrac{1}{x}\right) = \lambda - f\circ f_{k-3-j}\left(\dfrac{1}{x}\right) = \dfrac{1}{f_{k-3-j}\left(\frac{1}{x}\right)}\ ,
$$
so the result is proved by induction.
\end{proof}

\begin{lemma}
\label{endpoints}
For any $\ell\ge1$, the set $\E(\ell)$ of endpoints of the subdivision $\I(\ell)$ is invariant by $x\mapsto 1/x$.
Moreover, the largest finite element of $\E(\ell)$ is $\ell\lambda$ and the smallest positive one is $1/\ell\lambda$.
\end{lemma}
\begin{proof}
Recall that the elements of $\E(1)$ are of the form $b_j=f_{j-1}(\infty)=f_{j}(0)$, and the largest finite endpoint is $b_1=\lambda$. 
Hence, the result for $\ell=1$ is a direct consequence of Lemma~\ref{inverse}. 

Assume now that the result is true for $\ell\ge1$. Consider $b\in\E(\ell+1)\setminus\E(\ell)$. There exists $0\le j\le k-2$ and $b'\in\E(\ell)$ such that $b=f_j(b')$. Since $1/b'$ is also in $\E(\ell)$, we see from Lemma~\ref{inverse} that $1/b=f_{k-2-j}(1/b')\in\E(\ell+1)$. Hence $\E(\ell+1)$ is invariant by $x\mapsto 1/x$. Now, since $f_0$ is decreasing, the largest finite endpoint of $\E(\ell+1)$ is $f_0(1/\ell\lambda)=(\ell+1)\lambda$, and the smallest positive endpoint of $\I(\ell+1)$ is $1/(\ell+1)\lambda$.
\end{proof}

\begin{proof}[Proof of Proposition~\ref{finiteCF}]
 The set of nonnegative real numbers admitting a finite $\lambda$-Rosen continued fraction expansion is the smallest subset of $\RR_+$ containing $0$ which is invariant under $x\mapsto 1/x$ and $x\mapsto x+\lambda$. By Lemma~\ref{endpoints}, the set $\bigcup_{\ell\ge1}\E_\ell$ is invariant under $x\mapsto 1/x$. Moreover, it is also invariant by $x\mapsto f_{k-2}(x)=1/(x+\lambda)$, and contains $b_{k-1}=0$.
\end{proof}

\begin{remark}
The preceding proposition generalizes the well-known fact that the endpoints of Stern-Brocot intervals are the rational numbers, that is real numbers admitting a finite continued fraction expansion.
\end{remark}

\begin{proof}[Proof of Lemma~\ref{generation}]
This is a direct consequence of Proposition~\ref{finiteCF} and the fact that the set of numbers admitting a finite $\lambda$-Rosen continued fraction expansion is dense in $\RR$ for any $\lambda<2$ (see \cite{rosen1954}).
\end{proof}

\section{Coupling with a two-sided stationary process}
\label{Sec:Coupling}
If $|F_{n+1}/F_n|$ was a stationary sequence with distribution $\nu_{k, \rho}$, then a direct application of the ergodic theorem would give the convergence stated in Theorem~\ref{MainTheorem}. 
The purpose of this section is to prove via a coupling argument that everything goes as if it was the case.
For this, we embed the sequence $(X_n)_{n\ge 3}$ in a doubly-infinite i.i.d. sequence $(X_n^*)_{n\in\ZZ}$ with $X_n=X_n^*$ for all $n\ge 3$. 
We define the reduction of $(X^*)_{-\infty< j\le n}$, which gives a left-infinite sequence of i.i.d. blocks, and denote by $q_n^*$ the corresponding limit point, which follows the law $\nu_{k, \rho}$.
We will see that for $n$ large enough, the last $\ell$ blocks of $\red(X_3\dots X_n)$ and $\red((X^*)_{-\infty< j\le n})$ are the same. 
Therefore, the quotient $q_n\egdef F_n^r/F_{n-1}^r$ is well-approximated by $q_n^*$, and an application of the ergodic theorem to $q_n^*$ will give the announced result.

\subsection{Reduction of a left-infinite sequence}
\label{reductioninfinie}
We will define the reduction of a left-infinite i.i.d. sequence $(X^*)_{-\infty}^0$ by considering the successive reduced sequence $\red(X_{-n}^*\dots X_0^*)$.

\begin{prop}
\label{stability}
For all $\ell\ge 1$, there exists almost surely $N(\ell)$ such that the last $\ell$ blocks of $\red(X_{-n}^*\dots X_0^*)$ are the same for any $n\ge N(\ell)$.
\end{prop}

This allows us to define almost surely the reduction of a left-infinite i.i.d. sequence $(X^*)_{-\infty}^0$ as the left-infinite sequence of blocks obtained in the limit of $\red(X_{-n}^*\dots X_0^*)$ as $n\to\infty$. 

Let us call \emph{excursion} any finite sequence $w_1\dots w_m$ of $R$'s and $L$'s such that $\red (w_1\dots w_m)=\emptyset$. 
We say that a sequence is {\em proper} if its reduction process does not end with a deletion. This means that the next letter is not flipped during the reduction.

The proof of the proposition will be derived from the following lemmas.

\begin{lemma}
\label{Lem:excursion}
If there exists $n>0$ such that $X_{-n}^*\dots X_{-1}^*$ is not proper, then $X_0^*$ is preceded by a unique excursion.
\end{lemma}

\begin{proof}
We first prove that an excursion can never be a suffix of a strictly larger excursion.
Let $W=W_1RW'$ be an excursion, with $RW'$ another excursion. Then $WL=W_1RW'L=\pm R$ and $RW'L=\pm R$, which implies that $W_1=\pm\Id$. It follows that $\red(W_1)=\begin{pmatrix}
                         \pm 1 & 0 \\ 0 & \pm 1                                                                                     \end{pmatrix}$.
Observe that $\red(W_1)$ cannot start with $L$'s since $\red (W_1 RW')=\emptyset$. Therefore, it is a concatenation of $s$ blocks, corresponding to some function $f_{j_1}\circ\cdots  f_{j_s}$ which cannot be $x\mapsto\pm x$ unless $s=0$. 
But $s=0$ means that $\red(W_1)=\emptyset$, so $\red(W) = \red (LW')=\emptyset$, which is impossible.

Observe first that, if $X_0^*$ is not flipped during the reduction of $X_{-(n-1)}^*\dots X_0^*$ but is flipped during the reduction of $X_{-n}^*\dots X_0^*$, then $X_{-n}^*$ is an $R$ which is removed during the reduction process of $X_{-n}^*\dots X_0^*$. In particular, this is true if we choose $n$ to be the smallest integer such that $X_0^*$ is flipped during the reduction of $X_{-n}^*\dots X_0^*$. Therefore there exists $0\le j<n$ such that $X_{-n}^*\dots X_{-(j+1)}^*$ is an excursion. 
If $j=0$ we are done; 
otherwise the same observation proves that $X_{-j}^*$ is an $L$ which is flipped during the reduction process of $X_{-n}^*\dots X_{-j}^*$. Therefore, $X_0^*$ is flipped during the reduction of $RX_{-(j-1)}^*\dots X_0^*$, but not during the reduction of $X_{-\ell}^*\dots X_0^*$ for any $\ell \le j-1$.
Iterating the same argument finitely many times proves that $\red(X_{-n}^*\dots X_{-1}^*)=\emptyset$.
\end{proof}

\begin{lemma}
 $$\sum_{w \mbox{\begin{scriptsize} excursions\end{scriptsize}}} \PP(w) < 1 .$$
\end{lemma}

\begin{proof}
$X_0$ is an $R$ which does not survive during the reduction process if and only if it is the beginning of an excursion. 
By considering the longest such excursion, we get
$$ p(1-p_R) = \sum_{w \mbox{\begin{scriptsize} excursions\end{scriptsize}}} \PP(w) \Bigl[(1-p) p_R+p\Bigr]. $$
Hence,
\begin{equation}
\label{eq:excursion}
 \sum_{w \mbox{\begin{scriptsize} excursions\end{scriptsize}}} \PP(w) = \dfrac{p(1-p_R)}{(1-p) p_R+p} <1.
\end{equation}
\end{proof}

We deduce from the two preceding lemmas:
\begin{corollary}
\label{coro}
There is a positive probability that for all $n>0$ the sequence $X_{-n}^*\dots X_{-1}^*$ be proper.
\end{corollary}

\begin{proof}[Proof of Proposition~\ref{stability}]
We deduce from Corollary~\ref{coro} that with probability 1 there exist infinitely many $j$'s such that 
\begin{itemize}
 \item $X^*_{-j}$ is an $R$ which survives in the reduction of $X_{-j}^*\dots X_0^*$;
 \item $X_{-n}^*\dots X^*_{-j-1}$ is proper for all $n\ge j$.
\end{itemize}
For such $j$, the contribution of $X_{-j}^*\dots X_0^*$ to $\red(X_{-n}^*\dots X_0^*)$ is the same for any $n\ge j$.
\end{proof}

The same argument allows us to define almost surely $\red((X^*)_{-\infty}^{n})$ for all $n\in\ZZ$, which is a left-infinite sequence of blocks. Observe that we can associate to each letter of this sequence of blocks the time $t\le n$ at which it was appended. We number the blocks by defining $B_0^n$ as the rightmost block whose initial $R$ was appended at some time $t<0$. 
For $n>0$, we have
$\red((X^*)_{-\infty}^{n})=\ldots B_{-1}^nB_0^nB_1^n\ldots B_{L(n)}^n$ where $0\le L(n)\le n$. The random number $L(n)$ evolves in the same way as the number of $R's$ in $\red(X_3\ldots X_n)$. By Lemma~\ref{Survival}, $L(n)\to +\infty$ as $n\to \infty$ almost surely. As a consequence, for any $j\in\ZZ$ the block $B_j^n$ is well-defined and constant for all large enough $n$. We denote by $B_j$ the limit of $B_j^n$. The concatenation of these blocks can be viewed as the reduction of the whole sequence $(X^*)_{-\infty}^{+\infty}$.
The same arguments as those given in Section~\ref{reductionSection} prove that the blocks $B_j$ are i.i.d. with common distribution law $\PP_{\rho}$. 

It is remarkable that the same result holds if we consider only the blocks in the reduction of $(X^*)_{-\infty}^0$.

\begin{prop}
\label{distribution}
The sequence $\red((X^*)_{-\infty}^0)$ is a left-infinite concatenation of i.i.d. blocks with common distribution law $\PP_{\rho}$.
\end{prop}

\begin{proof}
Observe that $\red((X^*)_{-\infty}^{0})=\red((X^*)_{-\infty}^{L})$ where $L\le 0$ is the (random) index of the last letter not removed in the reduction process of $(X^*)_{-\infty}^0$. 
For any $\ell\le0$, we have $L=\ell$ if and only if $(X^*)_{-\infty}^\ell$ is proper and $(X^*)_{\ell +1}^{0}$ is an excursion.
For any bounded measurable function $f$, since $\EE\bigl[f(\red((X^*)_{-\infty}^{\ell}))\ \big|\ (X^*)_{-\infty}^{\ell} \mbox{ is proper}\bigr]$ does not depend on $\ell$, we have
\begin{eqnarray*}
\lefteqn{\EE\bigl[f(\red((X^*)_{-\infty}^{0})\bigr]}\\
&=& \sum_{\ell} \PP(L=\ell)\ \EE\bigl[f(\red((X^*)_{-\infty}^{\ell}))\ \big|\ L=\ell\bigr]\\
&=& \sum_{\ell} \PP(L=\ell)\ \EE\bigl[f(\red((X^*)_{-\infty}^{\ell}))\ \big|\ (X^*)_{-\infty}^{\ell} \mbox{ is proper, } (X^*)_{\ell+1}^0 \mbox{ is an excursion}\bigr]\\
&=& \sum_{\ell} \PP(L=\ell)\ \EE\bigl[f(\red((X^*)_{-\infty}^{\ell}))\ \big|\ (X^*)_{-\infty}^{\ell} \mbox{ is proper}\bigr]\\
&=& \EE\bigl[f(\red((X^*)_{-\infty}^{0}))\ \big|\ (X^*)_{-\infty}^{0} \mbox{ is proper}\bigr].
\end{eqnarray*}
This also implies that the law of $\red((X^*)_{-\infty}^{0})$ is neither changed when conditioned on the fact that $(X^*)_{-\infty}^{0}$ is not proper. 

Assume that $(X^*)_{-\infty}^{0}$ is proper. The fact that the blocks of $\red((X^*)_{-\infty}^{0})$ will not be subsequently modified in the reduction process of $(X^*)_{-\infty}^{\infty}$ only depends on $(X^*)_{1}^{\infty}$. Therefore, 
$\EE\bigl[f(\red((X^*)_{-\infty}^{0}))\ \big| \ (X^*)_{-\infty}^{0} \mbox{ is proper}\bigr]$ is equal to
$$\EE\bigl[f(\red((X^*)_{-\infty}^{0}))\ \big| \ (X^*)_{-\infty}^{0} \mbox{ is proper and blocks of } \red((X^*)_{-\infty}^{0}) \mbox{ are definitive} \bigr].
$$
The same equality holds if we replace ``proper'' with ``not proper''.
Hence, the law of $\red((X^*)_{-\infty}^{0})$ is the same as the law of $\red((X^*)_{-\infty}^{0})$ conditioned on the fact that blocks of $\red((X^*)_{-\infty}^{0})$ are definitive. But we know that definitive blocks are i.i.d. with common distribution law $\PP_{\rho}$.
\end{proof}

\subsection{Quotient associated to a left-infinite sequence}
Let $n$ be a fixed integer. For $m\ge0$, we decompose $\red((X^*)_{n-m< i\le n})$ into blocks $B_\ell, \ldots, B_1 = (RL^{j_\ell}), \ldots, (RL^{j_1})$, to which we associate the generalized Stern-Brocot interval $I_{j_1, j_2, \dots, j_{\ell}}$. 
If we let $m$ go to infinity, the preceding section shows that this sequence of intervals converges almost surely to a point $q_n^*$. By Proposition~\ref{distribution}, $q_n^*$ follows the law $\nu_{k, \rho}$.

Since $(q_n^*)$ is an ergodic stationary process, and $\log(\cdot)$ is in $L^1(\nu_{k, \rho})$, the ergodic theorem implies
\begin{equation}
\label{ergodic}
 \dfrac{1}{N} \sum_{n=1}^N \log q_n^* \tend{N}{\infty} \int_{\RR_+} \log q \, d\nu_{k, \rho} (q)\quad\mbox{almost surely.}
\end{equation}
The last step in the proof of the main theorem is to compare the quotient $q_n = F_n^r/F_{n-1}^r$ with $q_n^*$.
\begin{prop}
\label{comparison}
$$\dfrac{1}{N} \sum_{n=3}^N \bigl|\log q_n^* - \log |q_n|\bigr| \tend{N}{\infty} 0 \quad\mbox{almost surely.}$$
\end{prop}

We call \emph{extremal} the leftmost and rightmost intervals of $\I(\ell)$.

\begin{lemma}
\label{sup}
$$
s_\ell \egdef \sup_{\substack{I\in\I(\ell) \\ I\mbox{\scriptsize not extremal}}} \sup_{q, q^* \in I} |\log q^* - \log q| \tend{\ell}{\infty} 0
$$
\end{lemma}
\begin{proof}
 Fix $\varepsilon >0$, and choose an integer $M>1/\varepsilon$. By Lemma~\ref{generation}, since $\log(\cdot)$ is uniformly continuous on $[1/M\lambda, M\lambda]$, we have for $\ell$ large enough
$$
\sup_{\substack{I\in\I(\ell) \\ I\subset [1/M\lambda, M\lambda]}} \sup_{q, q^* \in I} |\log q^* - \log q| \le \varepsilon .
$$
If $I\in\I(\ell)$ is a non-extremal interval included in $[0, 1/M\lambda]$ or in $[M\lambda, +\infty]$, there exists an integer $j\in[M, \ell]$ such that $I\subset [1/(j+1)\lambda, 1/j\lambda]$ or $I\subset [(j+1)\lambda, j\lambda]$. Hence, 
$$
\sup_{q, q^* \in I} |\log q^* - \log q| \le \log\left( \dfrac{j+1}{j}\right) \le \log\left( 1+\dfrac{1}{M}\right)\le \varepsilon .
$$
\end{proof}

\begin{proof}[Proof of Proposition~\ref{comparison}]
For any $j\in\ZZ$, we define the following event $E_j$:
\begin{itemize}
 \item $X^*_{j}$ is an $R$ which survives in the reduction of $(X_{i}^*)_{i\ge j}$;
 \item $X_{i}^*\dots X^*_{j-1}$ is proper for all $i< j$.
\end{itemize}
Observe that if $E_j$ holds for some $j\ge3$, then for all $n\ge j$, 
\begin{eqnarray*}
\red(X_3\dots X_n) &=& \red(X_3\dots X_{j-1})\ \red(X_j\dots X_n)\\
\mbox{and}\quad\red((X^*)_{-\infty}^{n}) &=& \red((X^*)_{-\infty}^{j-1})\ \red(X^*_j\dots X^*_n).
\end{eqnarray*}
Hence, since $X_j\ldots X_n=X_j^*\ldots X_n^*$, they give rise in both reductions to the same blocks, the first one being definitive.
Since each $E_j$ holds with the same positive probability, the ergodic theorem yields
\begin{equation}
 \label{E}
 \dfrac{1}{n}\sum_{j=3}^n \ind{E_j}\tend{n}{\infty} \PP(E_3) >0\quad\mbox{almost surely,}
\end{equation} 
hence the number of definitive blocks of $\red(X_3\dots X_n)$ and of $\red((X^*)_{-\infty}^{n})$ which coincide grows almost surely linearly with $n$ as $n$ goes to~$\infty$ (these definitive blocks may be followed by some additional blocks which also coincide).

Recall the definition of $L_+$ given in Lemma~\ref{L+} and observe that for $n\ge n_{L_+}$, $q_n>0$. 
Observe also that, by definition of $I_{j_1, j_2, \dots, j_\ell}$, if $q$ and $q^*$ are two positive real numbers, $f_{j_1}\circ f_{j_2}\circ \cdots \circ f_{j_\ell}(q)$ and $f_{j_1}\circ f_{j_2}\circ \cdots \circ f_{j_\ell}(q^*)$ belong to the same interval of $\I(\ell)$. 

From~\eqref{E}, we deduce that, almost surely, for $n$ large enough, at least $L_+ +\sqrt{n}$ definitive blocks of $\red(X_3\dots X_n)$ and of $\red((X^*)_{-\infty}^{n})$ coincide (possibly followed by some additional blocks which also coincide).
This ensures that $q_n$ and $q_n^*$ belong to the same interval of the subdivision $\I(\sqrt{n})$.

By Lemma~\ref{sup}, it remains to check that, almost surely, there exist only finitely many $n$'s such that $q_n^*$ belongs to an extremal interval of the subdivision $\I(\sqrt{n})$. But this is a direct application of Borel-Cantelli Lemma, observing that the measure $\nu_{k, \rho}$ of an extremal interval of $\I(\ell)$ decreases exponentially fast with $\ell$.
\end{proof}

%

We now conclude the section by the proof of the convergence to the integral given in Theorem~\ref{MainTheorem}, linear case:
Since $F_n=\pm F_n^r$, we can write $n^{-1}\log |F_n|$ as
$$
\dfrac{1}{n} \log |F_2| + \dfrac{1}{n}\sum_{j=3}^n \log q_j^* + \dfrac{1}{n}\sum_{j=3}^n \left(\log |q_j|-\log q_j^*\right),
$$
and the convergence follows using Proposition~\ref{comparison} and \eqref{ergodic}.

\section{Positivity of the integral}
\label{Section:positivity}
We now turn to the proof of the positivity of $\gamma_{p,\lambda_k}$. It relies on the following lemma, whose proof is postponed.

\begin{lemma}
 \label{lemme:nu_rho}
Fix $0<\rho<1$. For any $t>0$, 
\begin{equation}\label{eq:positivity}
\Delta_t := \nu_{k, \rho}\left([t, \infty)\right) - \nu_{k, \rho}\left([0, 1/t]\right) \ge 0.
\end{equation}
Moreover, there exists $t>1$ such that the above inequality is strict.
\end{lemma}

Using Fubini's theorem, we obtain that $\gamma_{p,\lambda_k}$ is equal to
\begin{eqnarray*}
\int_0^\infty \log x\, d\nu_{k, \rho} (x)
&=& \int_1^\infty \log x\, d\nu_{k, \rho} (x)-\int_0^1 \log (1/x)\, d\nu_{k, \rho} (x)\\
&=& \int_0^\infty \nu_{k, \rho} ([e^{u},\infty)) du - \int_0^\infty \nu_{k, \rho} ([0, e^{-u}]) du 
\end{eqnarray*}
which is positive if $0<\rho<1$ by Lemma~\ref{lemme:nu_rho}. 
Thus, $\gamma_{p,\lambda_k}>0$ for any $p>0$. 
This ends the proof of Theorem~\ref{MainTheorem}, linear case.

\begin{proof}[Proof of Lemma~\ref{lemme:nu_rho}]
By Lemma~\ref{generation}, it is enough to prove the lemma when $t$ is the endpoint of an interval of the subdivision $\I(\ell)$. 
This is done by induction on $\ell$. Obviously, $\Delta_0=\Delta_\infty=0$.
When $\ell=1$ and $\ell=2$, if $t\neq 0,\infty$, it can be written as $f_j(b_i)$ for $0\le j\le k-2$ and $0\le i\le k-2$, and we get $1/t=f_{k-2-j}(b_{k-1-i})$ (see Lemma~\ref{inverse}).
Setting $Z:=\sum_{s=0}^{k-2}\rho^s$, we have 
$$
\nu_{k, \rho}\left([t, \infty)\right) = \sum_{s=0}^{j-1} \nu_{k, \rho}\left([b_{s+1}, b_s)\right) +\nu_{k, \rho}\left([t, b_{j})\right)
=\sum_{s=0}^{j-1}\frac{\rho^s}{Z} + \frac{\rho^j}{Z} \nu_{k, \rho}\left([0,b_i]\right)= \sum_{s=0}^{j-1}\frac{\rho^s}{Z} + \frac{\rho^j}{Z}\sum_{s=i}^{k-2}\frac{\rho^s}{Z}.
$$
Therefore,
\begin{eqnarray*}
\lefteqn{\nu_{k, \rho}\left([t, \infty)\right) - \nu_{k, \rho}\left([0, 1/t]\right)} \\
&=& \sum_{s=0}^{j-1}\frac{\rho^s}{Z} + \frac{\rho^j}{Z}\sum_{s=i}^{k-2}\frac{\rho^s}{Z} 
-\left( \sum_{s=k-1-j}^{k-2}\frac{\rho^s}{Z} + \frac{\rho^{k-2-j}}{Z} \ \sum_{s=0}^{k-2-i}\frac{\rho^s}{Z}\right)\\
&=& \sum_{s=0}^{j-1}\frac{\rho^s}{Z}\left(1-\rho^{k-1-j}\right) + \frac{1}{Z}\left(\rho^{i+j}-\rho^{k-2-j}\right)\sum_{s=0}^{k-2-i}\frac{\rho^s}{Z}.
\end{eqnarray*}
Since $i\le k-2$, we have $\rho^{i+j}-\rho^{k-2-j}\ge \rho^{k-2-j}(\rho^{2j}-1)$. Moreover, $\sum_{s=0}^{k-2-i}\frac{\rho^s}{Z}\le 1$. Thus, 
$$
Z\Delta_t \ge 
\sum_{s=0}^{j-1}\rho^s\left(1-\rho^{k-1-j}\right) - \rho^{k-2-j}(1-\rho^{2j}).
$$
Observe that $(1-\rho^{k-1-j}) = (1-\rho)\sum_{s=0}^{k-2-j}\rho^s$ and that $1-\rho^{2j}=(1+\rho^{j})(1-\rho)\sum_{s=0}^{j-1}\rho^s$. 
Hence, 
$$
Z\Delta_t
\ge (1-\rho) \sum_{s=0}^{j-1}\rho^s \left(\sum_{s=0}^{k-2-j}\rho^s - \rho^{k-2-j}(1+\rho^{j})\right),
$$
which is positive as soon as $j<k-2$. 
The quantity $\Delta_t$ is invariant when $t$ is replaced by $1/t$, so we also get the desired result for $j=k-2$.

Assume~\eqref{eq:positivity} is true for any endpoint of intervals of the subdivision $\I(j)$, $j\le\ell-1$. 
Let $t$ be an endpoint of an interval of $\I(\ell)$; then there exists an interval $[t_1, t_2]$ of $\I(\ell-2)$ such that $t\in[t_1, t_2]$. We can write 
\begin{eqnarray*}
\nu_{k, \rho}\left([t, \infty)\right) &=& \nu_{k, \rho}\left([t_2, \infty)\right) + \nu_{k, \rho}\left([t_1, t_2]\right) \nu_{k, \rho}\left([u, \infty)\right)\\
\mbox{and }\nu_{k, \rho}\left([0, 1/t]\right) &=& \nu_{k, \rho}\left([0, 1/t_2]\right) + \nu_{k, \rho}\left([1/t_2, 1/t_1]\right)  \nu_{k, \rho}\left([0, 1/u]\right)
\end{eqnarray*}
for some endpoint $u$ of an interval of $\I(2)$. 
If $\nu_{k, \rho}\left([t_1, t_2]\right)\ge \nu_{k, \rho}\left([1/t_2, 1/t_1]\right)$, we get the result since~\eqref{eq:positivity} holds for $u$, and $t_2$. 
Otherwise, we can write $\Delta_t$ as
$$ \Delta_{t_1}
- \nu_{k, \rho}\left([t_1, t_2]\right) + \nu_{k, \rho}\left([1/t_2, 1/t_1]\right)
+ \nu_{k, \rho}\left([t_1, t_2]\right) \nu_{k, \rho}\left([u, \infty)\right)-\nu_{k, \rho}\left([1/t_2, 1/t_1]\right)  \nu_{k, \rho}\left([0, 1/u]\right) 
$$
which is greater than 
$$ 
\Delta_{t_1} + \nu_{k, \rho}\left([t_1, t_2]\right) \Delta_u\ge 0.
$$
\end{proof}

\begin{remark}
\label{rho = 1}
 We can also define the probability measure $\nu_{k,\rho}$ for $\rho=1$. (When $k=3$, this is related to Minkowski's Question Mark Function, see \cite{denjoy1938}.)
It is straightforward to check that $\nu_{k,1}\left([t, \infty)\right) - \nu_{k, 1}\left([0, 1/t]\right) = 0$ for all $t>0$, which yields 
$$
\int_0^\infty \log x\, d\nu_{k, 1} (x) = 0.
$$
\end{remark}

\section{Reduction: The non-linear case}
\label{Section:non-linear}
In the non-linear case, where 
$\widetilde F_{n+2} = |\lambda \widetilde F_{n+1} \pm \widetilde F_{n}|$, the sequence $(\widetilde F_n)_{n\ge1}$ can also be coded by the sequence $(X_n)_{n\ge 3}$ of i.i.d. random variables taking values in the alphabet $\{R, L\}$ with probability $(p, 1-p)$. 
Each $R$ corresponds to choosing the $+$ sign and can be interpreted as the right multiplication of $(\widetilde F_{n-1}, \widetilde F_n)$ by the matrix $R$ defined in~\eqref{matrices}. Each $L$ corresponds to choosing the $-$ sign but the interpretation in terms of matrices is slighty different, since we have to take into account the absolute value: $X_{n+1}=L$ corresponds either to the right multiplication of $(\widetilde F_{n-1}, \widetilde F_n)$ by $L$ if $(\widetilde F_{n-1}, \widetilde F_n)L$ has nonnegative entries, or to the multiplication by 
\begin{equation}
\label{matrices2}
L' \egdef 
\begin{pmatrix}
0 & 1\\
1 & -\lambda
\end{pmatrix}.
\end{equation}
Observe that for all $0\le j\le k-2$, the matrix $RL^j$ has nonnegative entries (see~\eqref{RtimesPowersOfL}), whereas 
$RL^{k-1} = 
\begin{pmatrix}
1 & 0\\
0 & -1
\end{pmatrix}$. 
Therefore, if $X_i=R$ is followed by some $L$'s, we interpret the first $(k-2)$ $L$'s as the right multiplication by the matrix $L$, whereas the $(k-1)$-th $L$ corresponds to the multiplication by $L'$. 
Moreover, $RL^{k-2}L' = \Id$, so we can remove all patterns $RL^{k-1}$ in the process $(X_n)$.

\newcommand{\tred}{\widetilde{\red}}
We thus associate to $x_3\ldots x_n$ the word $\tred(x_3\ldots x_n)$, which is obtained by the same reduction as $\red(x_3\ldots x_n)$, except that the letter added in Step 1 is always $x_{i}$. 
We have
$$ (\widetilde F_{n-1},\widetilde F_n) = (\widetilde F_1,\widetilde F_2) \tred(x_3\ldots x_n). $$

Since the reduction process is even easier in the non-linear case, we will not give all the details but only insist on the differences with the linear case. 
The first difference is that the survival probability of an $R$ is positive only if $p>1/k$.

\begin{lemma}
\label{SurvivalNL}
For $p > 1/k$, the number of $R$'s in $\tred(X_3\ldots X_n)$ satisfies
$$
|\tred(X_3\ldots X_n)|_R\tend{n}{\infty}+\infty\qquad\mbox{a.s.}
$$
and the survival probability $p_R$ is for $p<1$ the unique solution in $]0, 1]$ of 
\begin{equation}
\label{pRNL}
\tilde g(x)=0,\quad \mbox{ where }\quad  \tilde g(x)\egdef(1-x)\left(1+\frac{p}{1-p}\ x\right)^{k-1} - 1\ . 
\end{equation}
If $p\le 1/k$, $p_R=0$.
\end{lemma}
\begin{proof}
Since each deletion of an $R$ goes with the deletion of $(k-1)$ $L$'s, if $p>1/k$, the law of large numbers ensures that the number of remaining $R$'s goes to infinity. If $p<1/k$, there only remains $L$'s, so $p_R=0$. 

Doing the same computations as in Section~\ref{section:survival}, we obtain that, for all $0\le j\le k-2$, the probability $p_j$ for an $R$ to be followed by $L^jR$ after the subsequent steps of the reduction is 
$$p_j=\frac{(1-p)^j p p_R}{(1-p+pp_R)^{j+1}}.$$
Since $p_R=\sum_{j=0}^{k-2}p_j$, we get that $p_R$ is solution of $\tilde g(x)=0$. 
Observe that $\tilde g(0)=0$, $\tilde g(1)=-1$, $\tilde g'(0)>0$ for $p>1/k$ and $\tilde g'$ vanishes at most once on $\RR_+$.
Hence, for $p>1/k$, $p_R$ is the unique solution of $\tilde g(x)=0$ in $]0, 1]$.
For $p=1/k$, $\tilde g'(0)=0$ and the unique nonnegative solution is $p_R=0$.
\end{proof}

\subsection{Case $p>1/k$}
As in the linear case, the sequence of surviving letters 
$$ (S_j)_{j\ge3} = \lim_{n\to\infty} \tred(X_3\ldots X_n) $$
is well defined for $p>1/k$, and can be written as the concatenation of a certain number $s\ge 0$ of starting $L$'s and of blocks: 
$$
S_1S_2 \ldots = L^s B_1B_2\ldots
$$
where for all $\ell\ge1$, $B_\ell\in\{R, RL, \ldots, RL^{k-2}\}$. 
These blocks appear with the same distribution $\PP_{\rho}$ as in the linear case, but with a different parameter $\rho$.

\begin{lemma}
 \label{lawNL}
In the non-linear case, for $p>1/k$, the blocks $(B_\ell)_{\ell\ge1}$ are i.i.d. with common distribution law $\PP_{\rho}$ defined by~\eqref{defPrho}, where $\rho \egdef \sqrt[k-1]{1-p_R}$ and $p_R$ is given by Lemma~\ref{SurvivalNL}.
\end{lemma}
As in Section~\ref{reductioninfinie}, we can embed the sequence $(X_n)_{n\ge 3}$ in a doubly-infinite i.i.d. sequence $(X_n^*)_{n\in\ZZ}$ with $X_n=X_n^*$ for all $n\ge 3$. 
We define the reduction of $(X^*)_{-\infty< j\le n}$ by considering the successive $\tred(X_{n-N}\dots X_n)$. 
The analog of Proposition~\ref{stability} is easier to prove than in the linear case since the deletion of a pattern $RL^{k-1}$ does not affect the next letter. The end of the proof is similar.

\subsection{Case $p\le 1/k$}
\label{Sec:p=0}
Since in this case the survival probability of an $R$ is $p_R=0$, the reduced sequence $\tred(X_0^{\infty})$ contains only $L$'s. 
We consider the subsequence $(\widetilde F_{n_j})$ where $n_j$ is the time when the $j$-th $L$ is appended to the reduced sequence. 
This subsequence satisfies, for any $j$, $\widetilde F_{n_{j+1}} = |\lambda \widetilde F_{n_{j}}-\widetilde F_{n_{j-1}}|$, which corresponds to the non-linear case for $p=0$. 

Therefore, we first concentrate on the deterministic sequence $\widetilde F_{n+1}= |\lambda \widetilde F_{n}-\widetilde F_{n-1}|$, with given nonnegative initial values $\widetilde F_0$ and $\widetilde F_1$.

\begin{prop}
\label{bounded}
 For any choice of $\widetilde F_0\ge 0$ and $\widetilde F_1\ge 0$, the sequence defined inductively by $\widetilde F_{n+1}= |\lambda \widetilde F_{n}-\widetilde F_{n-1}|$ is bounded.
\end{prop}

Lemma~\ref{reverse} in the next section gives a proof of this proposition for the specific case $\lambda =2\cos\pi/k $. 
We give here another proof based on a geometrical interpretation, which can be applied for any $0<\lambda <2$.

The key argument relies on the following observation: Let $\theta$ be such that $\lambda=2\cos\theta$. 
Fix two points $P_0,P_1$ on a circle centered at the origin $O$, such that the oriented angle $(OP_0,OP_1)$ equals $\theta$. 
Let $P_2$ be the image of $P_1$ by the rotation of angle $\theta$ and center $O$. Then the respective abscissae $x_0$, $x_1$ and $x_2$ of $P_0$, $P_1$ and $P_2$ satisfy $x_2=\lambda x_1 - x_0$. 
We can then geometrically interpret the sequence $(\widetilde F_n)$ as the successive abscissae of points in the plane. 

\begin{lemma}[Existence of the circle]
  Let $\theta\in ]0,\pi[$. For any choice of $(x, x')\in\RR_+^2\setminus\{(0,0)\}$, their exist a unique $R>0$ and two points $M$ and $M'$, with respective abscissae $x$ and $x'$, lying on the circle with radius $R$ centered at the origin, such that the oriented angle $(OM,OM')$ equals~$\theta$.
\end{lemma}

\begin{proof}
Assume that $x>0$. We have to show the existence of a unique $R$ and a unique $t\in ]-\pi/2,\pi/2[$ (which represents the argument of $M$) such that 
$$ R\cos t=x\quad\mbox{and}\quad R\cos(t+\theta)=x'. $$
This is equivalent to 
$$R\cos t=x \quad\mbox{and}\quad  \cos\theta - \tan t\ \sin \theta = \dfrac{x'}{x},$$
which obviously has a unique solution since $\sin \theta\neq 0$.

If $x=0$, the unique solution is clearly  $R=x'/\cos(\theta-\pi/2)$ and $t=-\pi/2$.

\emph{Remark:} Since $x_1>0$, we have $t+\theta<\pi/2$.
\end{proof}

\begin{proof}[Proof of Proposition~\ref{bounded}]
At step $n$, we interpret $\widetilde F_{n+1}$ in the following way: Applying the lemma with $x=\widetilde F_{n-1}$ and $x'=\widetilde F_n$, we find a circle of radius $R_n>0$ centered at the origin and two points $M$ and $M'$ on this circle with abscissae $x$ and $x'$. Consider the image of $M'$ by the rotation of angle $\theta$ and center $O$. If its abscissa is nonnegative, it is equal to $\widetilde F_{n+1}$, and we will have $R_{n+1}=R_n$. Otherwise, we have to apply also the symmetry with respect to the origin to get a point with abscissa $\widetilde F_{n+1}$. The circle at step $n+1$ may then have a different radius, but we now show that the radius always decreases (see Figure~\ref{Fig:cercle}).

\begin{figure}
\input{cercle.pstex_t}
\caption{$R_{n}=R_{n+1}$ is the radius of the largest circle, and $R_{n+2}$ is the radius of the smallest.}
\label{Fig:cercle}
\end{figure}
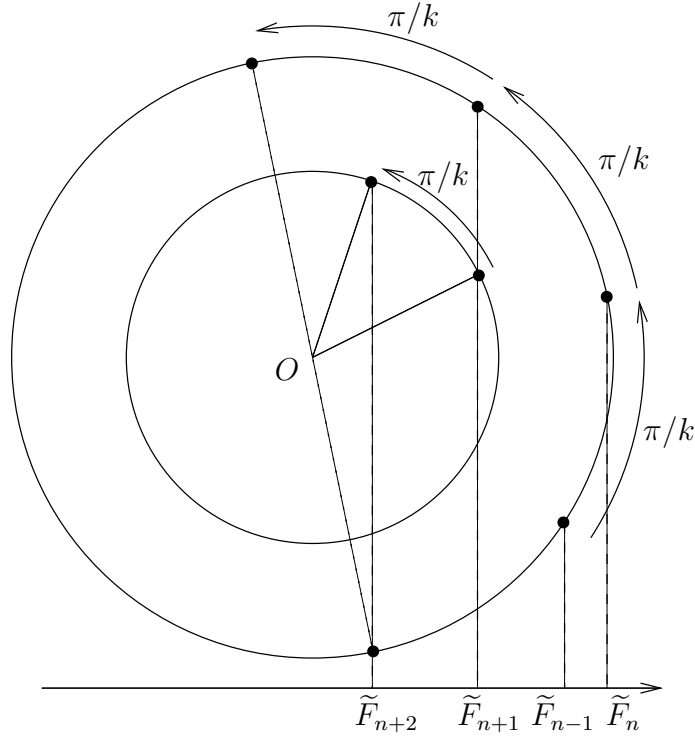

Indeed, denoting by $\alpha$ the argument of $M'$, we have in the latter case $\pi/2-\theta<\alpha\le \pi/2$, $\widetilde F_{n}=R_n\cos \alpha$ and $\widetilde F_{n+1}=R_n\cos (\alpha+\theta+\pi)> 0$.
At step $n+1$, we apply the lemma with $x=R_n\cos \alpha$ and $x'=R_n\cos (\alpha+\theta+\pi)$. 
From the proof of the lemma, if $\widetilde F_{n}=0$ (\textit{i.e.} if $\alpha=\pi/2$), $R_{n+1}=R_n\cos (\alpha+\theta+\pi)/\cos(\theta-\pi/2)=R_n$. 
If $\widetilde F_{n}>0$, we have $R_{n+1}=R_n\cos \alpha/\cos t$, where $t$ is given by 
$$
\cos\theta - \tan t\ \sin \theta = \dfrac{\cos (\alpha+\theta+\pi)}{\cos \alpha} = -(\cos\theta - \tan \alpha\ \sin \theta).
$$
We deduce from the preceding formula that $\tan t + \tan\alpha = 2\cos\theta/\sin\theta>0$, which implies $t>-\alpha$. 
On the other hand, as noticed at the end of the proof of the preceding lemma, $t+\theta<\pi/2$, hence $t<\alpha$.
Therefore, $\cos \alpha<\cos t$ and $R_{n+1}<R_n$.

Since $\widetilde F_n\le R_n\le R_1$ for all $n$, the proposition is proved. 
\end{proof}

We come back to the specific case $\lambda=2\cos\pi/k$.
\begin{prop}
\label{p=0}
Let $(\widetilde F_{n})$ be inductively defined by $\widetilde F_{n+1}= |\lambda \widetilde F_{n}-\widetilde F_{n-1}|$ and its two first positive terms.
The following properties are equivalent:
\begin{enumerate}
\item $\widetilde F_0/\widetilde F_1$ admits a finite $\lambda$-continued fraction expansion.
\item The sequence $(\widetilde F_{n})$ is ultimately periodic.
\item There exists $n$ such that $\widetilde F_{n}=0$.
\end{enumerate}
\end{prop}

\begin{proof}
We easily see from the proof of Proposition~\ref{bounded} that (2) and (3) are equivalent. 
We now prove that (3) implies (1) by induction on the smallest $n$ such that $\widetilde F_{n}=0$. 
If $\widetilde F_{2}=0$, then $|\lambda \widetilde F_{1}-\widetilde F_{0}|=0$, and we get $\widetilde F_0/\widetilde F_1=\lambda$. 
Let $n>2$ be the smallest $n$ such that $\widetilde F_{n}=0$. By the induction hypothesis, $\widetilde F_1/\widetilde F_2$ admits a finite $\lambda$-continued fraction expansion. 
Therefore, 
$$\frac{\widetilde F_0}{\widetilde F_1}= \lambda\pm \frac{1}{\widetilde F_1/\widetilde F_2}$$
admits a finite $\lambda$-continued fraction expansion. 

It remains to prove that (1) implies (3). 
We know from Proposition~\ref{finiteCF} that all positive real numbers that admit a finite $\lambda$-continued fraction expansion are endpoints of generalized Stern-Brocot intervals, hence by~\eqref{fj}, can be written as $[1, a_1, \dots, a_j]_{\lambda}$ with $a_i=\pm1$ for any $i$ and such that we never see more than $(k-1)$ alternated $\pm1$ in a row. We call such an expansion a \emph{standard expansion}.
Conversely, all real numbers that admit a standard expansion are endpoints of generalized Stern-Brocot intervals, hence are nonnegative.
Assume (1) is true. If $\widetilde F_0/\widetilde F_1 = [1]_{\lambda}$, then $\widetilde F_2=0$.
Otherwise, let $[1, a_1, \dots, a_j]_{\lambda}$ be a standard expansion of $\widetilde F_0/\widetilde F_1$.
Then, 
$$
\frac{\widetilde F_1}{\widetilde F_2} = \frac{1}{|\lambda - \widetilde F_0/\widetilde F_1|} 
= \Bigl|[a_1, \dots, a_j]_{\lambda}\Bigr|.
$$
If $a_1=1$, then $[a_1, \dots, a_j]_{\lambda}\ge0$ and it is equal to $\widetilde F_1/\widetilde F_2$. 
Otherwise, $\widetilde F_1/\widetilde F_2 = [-a_1, -a_2,\dots, -a_j]_{\lambda}$. 
In both cases, we obtain a standard expansion of $\widetilde F_1/\widetilde F_2$ of smaller size.
The result is proved by induction on $j$.
\end{proof}

\begin{remark}
In general, if $\widetilde F_0/\widetilde F_1$ does not admit a finite $\lambda$-continued fraction expansion, $(\widetilde F_{n})$ decreases exponentially fast to $0$. However, the exponent depends on the ratio $\widetilde F_0/\widetilde F_1$.
\end{remark}
We exhibit two examples of such behavior.

Let $q:=(\lambda + \sqrt{\lambda^2+4})/2$ be the fixed point of $f_0$. Start with $\widetilde F_0/\widetilde F_1 = q$. 
Then, by a straightforward induction, we get that for all $n\ge0$, $\widetilde F_n = q^{-n}\widetilde F_0$.

Start now with $\widetilde F_0/\widetilde F_1 = q'$, where $q'$ is the fixed point of $f_1$. 
Then, we easily get that for all $n\ge0$, $\widetilde F_{2n} = (q'f_0(q'))^{-n}\widetilde F_0$ and $\widetilde F_{2n+1} = \widetilde F_{2n}/q'$.
The exponent is thus $1/\sqrt{q'f_0(q')}$, which is different from $1/q$: For $k=3$, $q=\phi$ (the golden ratio) and $\sqrt{q'f_0(q')}=\sqrt{\phi}$. 

\begin{proof}[Proof of Theorem~\ref{Theorem2}]
We have seen that the subsequence $(\widetilde F_{n_j})$, where $n_j$ is the time when the $j$-th $L$ is appended to the reduced sequence, satisfies, $\widetilde F_{n_{j+1}} = |\lambda \widetilde F_{n_{j}}-\widetilde F_{n_{j-1}}|$ for any $j$.
From Proposition~\ref{bounded}, this subsequence is bounded.
Moreover, we can write $n_j=j+kd_j$, where $d_j$ is the number of $R$'s up to time $n_j$. 
By the law of large numbers, $d_j/n_j\to p$, and we get $j/n_j\to 1-kp$. This achieves the proof of Theorem~\ref{Theorem2}.
\end{proof}

\section{Case $\lambda \ge2$}
\label{Sec:lambda2}

The case $\lambda\ge2$ ($p>0$) is even easier to study since there is no reduction process. 

Observe that the linear and the non-linear case are essentially the same. 
Indeed, in the non-linear case, $\PP(\widetilde F_{n+1}/\widetilde F_n\ge 1| \widetilde F_{n-1}, \widetilde F_n)\ge p$ and if $\widetilde F_{n+1}/\widetilde F_n\ge 1$, then $\widetilde F_{n+2}/\widetilde F_{n+1}\ge 1$. 
Therefore, with probability $1$, there exists $N_+$ such that for all $n\ge N_+$, the quotients $\widetilde F_{n+1}/\widetilde F_n$ are larger than $1$. 
Moreover, for $n\ge N_+$, there is no need to take the absolute value and the sequence behaves like in the linear case. 
We thus concentrate on the linear case. 

We now fix $\lambda\ge 2$. The sequence of quotients $Q_n\egdef F_{n}/F_{n-1}$ is a real-valued Markov chain with probability transitions
$$
\PP\left( Q_{n+1} = f_R(q) \Big| Q_{n} = q \right) = p 
\quad\mbox{ and }\quad
\PP\left( Q_{n+1} = f_L(q) \Big| Q_{n} = q \right) = 1-p, 
$$
where $f_R(q)\egdef \lambda +1/q$ and $f_L(q)\egdef \lambda  -1/q$.

Let $B\egdef\dfrac{\lambda+\sqrt{\lambda^2-4}}{2}\in [1, \lambda]$ be the largest fixed point of $f_L$. 
Note that we have $\PP(Q_{n+1}\ge \lambda | Q_n)\ge\min(p, 1-p)$ for any $n\ge 2$ and, again, if $Q_n\ge B$, then $Q_{n+1}\ge B$. 
Thus, with probability $1$, there exists $N_+$ such that for all $n\ge N_+$, the quotients $Q_{n}$ are larger than $B$. 
Without loss of generality, we can henceforth assume that the initial values $a$ and $b$ are such that $Q_2\ge B$.

We inductively define sub-intervals of $\RR_+$ indexed by finite sequences of $R$'s and $L$'s:
$$
I_R\egdef f_R([B, \infty])=\left[\lambda, \lambda+\frac{1}{B} \right] \quad \mbox{ and }\quad I_L\egdef f_L([B, \infty]) = [B, \lambda], 
$$
and for any finite sequence $X$ in $\{R, L\}^*$, 
$$
I_{XR}\egdef f_R(I_X)\quad \mbox{ and }\quad I_{XL}\egdef f_L(I_X).
$$
Obviously, all these intervals are included in $\left[B,\lambda+\frac{1}{B}\right]$. 

\begin{lemma}\label{lem:I_W}
Let $W$ and $W'$ be two finite words in $\{R, L\}^*$. 
\begin{itemize}
\item If $W$ is a suffix of $W'$, then $I_{W'}\subset I_W$;
\item If neither $W$ is a suffix of $W'$ nor $W'$ is a suffix of $W$, then $I_W$ and $I_{W'}$ have disjoint interiors.
\end{itemize}
\end{lemma}

\begin{proof}
The first assertion is an easy consequence of the definition of $I_W$. 
To prove the second one, consider the largest common suffix $S$ of $W$ and $W'$. 
Since $LS$ and $RS$ are suffix of $W$ and $W'$, by the first assertion, it is enough to prove that $I_{LS}$ and $I_{RS}$ have disjoint interiors. This can be shown by induction on the length of $S$, using the fact that $f_R$ and $f_L$ are monotonic on $[B, \infty]$.
\end{proof}

\begin{lemma}\label{lem:singleton}
Let $(W_i)_{i\ge1}$ be a sequence of $R$'s and $L$'s. 
Then $\bigcap_{n\ge1} I_{W_n \dots W_1}$ is reduced to a single point.
\end{lemma}
\begin{proof}
By Lemma~\ref{lem:I_W}, $I_{W_{n+1} W_n \dots W_1}\subset I_{W_n \dots W_1}$. 
Since the intervals are compact and nonempty, their intersection is nonempty. 
It remains to prove that their length goes to zero.
First consider the case $\lambda >2$. The derivatives of $f_L$ and $f_R$ are of modulus less than $1/B^2<1$. 
Therefore, the length of $I_{W_n \dots W_1}$ is less than a constant times $(1/B^2)^{n}$.
Let us turn to the case $\lambda=2$. 
Observe that $I_{L^j}=\left[1, \frac{j+1}{j}\right]$, which is of length $1/j$. 
Hence, if $W_n \dots W_1$ contains $j$ consecutive $L$'s, then $I_{W_n \dots W_1}$ is included, for some $r<n$, in 
$I_{L^jW_r\dots W_1}=f_{W_1}\circ\dots\circ f_{W_r}(I_{L^j})$ which is of length less than $1/j$ 
(recall that the derivatives of $f_L$ and $f_R$ are of modulus less than $1$).
On the other hand, the derivatives of $f_L\circ f_R$ and $f_R\circ f_R$ are of modulus less than $1/(2B+1)^2=1/9$ on $[B, \infty]$.
Therefore, considering the maximum number of consecutive $L$'s in $W_n \dots W_1$, we obtain 
$\sup_{W_n \dots W_1}|I_{W_n \dots W_1}|\tend{n}{\infty} 0$.
\end{proof}

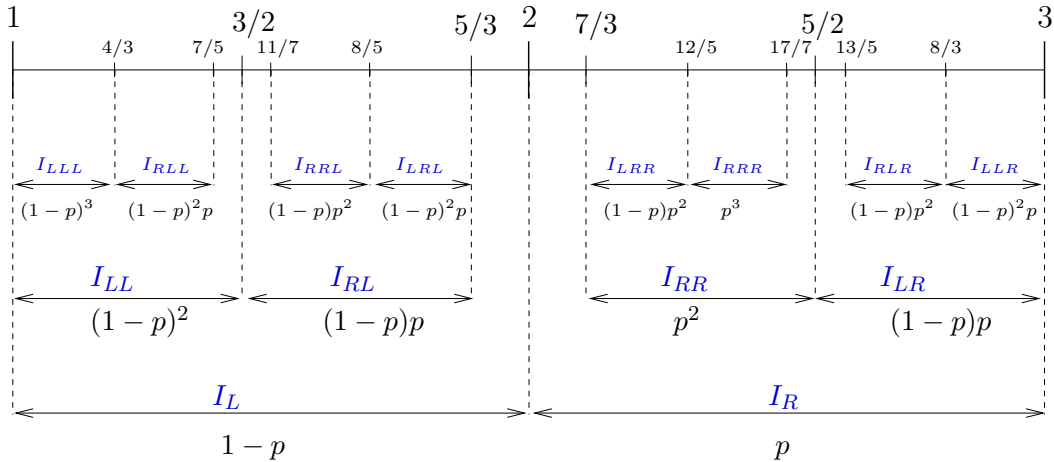
\begin{figure}[h]
	\begin{center}
	\input{mesure2.pstex_t}
	\end{center}
\caption{First stages of the construction of the measure $\mu_{p, 2}$.}
\label{mesure2}
\end{figure}

We deduce from the preceding results the invariant measure of the Markov chain $(Q_n)$.

\begin{corollary}
The unique invariant probability measure $\mu_{p, \lambda}$ of the Markov chain $(Q_n)=(F_n/F_{n-1})$ is given by 
\begin{equation}\label{nu}
 \mu_{p, \lambda}\left( I_{W}\right):=p^{|W|_R}(1-p)^{|W|_{L}}
\end{equation}
for any finite word $W$ in $\{R, L\}^*$, where $|W|_R$ and $|W|_L$ respectively denote the number of $R$'s and $L$'s in $W$.
\end{corollary}

We can now conclude the proof of Theorem~\ref{th:case_2} by invoking a classical theorem about law of large numbers for Markov chain (see \textit{e.g.}~\cite{meyn1993}, Theorem~17.0.1). 

\medskip

Note that the explicit form of the invariant measure when $p=1/2$ and $\lambda\ge 2$ was already given by Sire and Krapivsky~\cite{krapivsky2001}.

\section{Variations of the Lyapunov exponents}

\subsection{Variations with $p$}
\label{Sec:croissance_p}

\begin{theo}
\label{croissance}
For any integer $k\ge3$,
the function $p\mapsto\widetilde\gamma_{p,\lambda_k}$ is increasing and analytic on $]1/k,1[$, and 
the function $p\mapsto\gamma_{p,\lambda_k}$ is increasing and analytic on $]0,1[$.
Moreover,
\begin{equation}
 \label{limit}
\lim_{p\to 0}\gamma_{p,\lambda_k} = \lim_{p\to 1/k}\widetilde\gamma_{p,\lambda_k} = 0,
\end{equation}
and
\begin{equation}
 \label{limit1}
\lim_{p\to 1}\gamma_{p,\lambda_k} = \gamma_{1,\lambda_k} = \lim_{p\to 1}\widetilde\gamma_{p,\lambda_k} = \widetilde\gamma_{1,\lambda_k} = \log \left( \dfrac{\lambda_k+\sqrt{\lambda_k^2+4}}{2} \right).
\end{equation}
For any $\lambda\ge2$, the function $p\mapsto\gamma_{p,\lambda}$ is increasing and analytic on $]0,1[$.
\end{theo}

The proof of the theorem relies on the following proposition, whose proof is postponed to the end of the section.

\begin{prop}
\label{compare}
Let $(X_i)$ be a sequence of letters in the alphabet $\{R, L\}$ and $(X'_i)$ be a sequence of letters in the alphabet $\{R, L\}$ obtained from $(X_i)$ by turning an $L$ into an $R$. 
If $\lambda=\lambda_k$ for some $k\ge3$, then, in the non-linear case, any label $\widetilde F_n$ coded by the sequence $(X_i)$ is smaller than the corresponding label $\widetilde F'_n$ coded by $(X'_i)$.
If $\lambda\ge 2$, and if $F_2/F_1\ge 1$, any label $F_n$ coded by the sequence $(X_i)$ is smaller than the corresponding label $F'_n$ coded by $(X'_i)$.
\end{prop}

\begin{proof}[Proof of Theorem~\ref{croissance}] 
Let $\lambda=\lambda_k$ for some integer $k\ge3$.
Let $1/k<p\le p'\le1$. Let $(X_i)$ (respectively $(X'_i)$) be a sequence of i.i.d. random variables taking values in the alphabet $\{R, L\}$ with probability $(p, 1-p)$ (respectively $(p', 1-p')$).
We can realize a coupling of $(X_i)$ and $(X'_i)$ such that for any $i$, $X_i=R$ implies $X'_i=R$. 
From Proposition~\ref{compare}, it follows that the label $\widetilde F_n$ coded by $(X_i)$ is always smaller than the label $\widetilde F'_n$ coded by $(X'_i)$. 
We get that
$$
\widetilde\gamma_{p,\lambda_k} = \lim \dfrac{1}{n}\log \widetilde F_n
\le \lim\dfrac{1}{n}\log \widetilde F'_n = \widetilde\gamma_{p',\lambda_k}.
$$
Therefore, $p\mapsto \widetilde \gamma_{p,\lambda_k}$ is a non-decreasing function on $[1/k,1]$.

Observe that $p\mapsto p_R$ is non-decreasing in both (linear and non-linear) cases. 
Hence, the function $\rho : p\mapsto \sqrt[k-1]{1-p_R}$ is non-increasing in both cases.
We conclude that $p\mapsto\gamma_{p,\lambda_k}$ is non-decreasing on $[0,1]$. 

Since $\gamma_{p,\lambda_k}>0$ for $0<p<1$, the upper Lyapunov exponent associated to the product of random matrices is simple, and we know from \cite{peres1990} that $\gamma_{p,\lambda_k}$ is an analytic function of $p\in]0, 1[$, thus it is increasing. Via the dependence on $\rho$ which is an analytic function of $p$, we get that $\widetilde\gamma_{p,\lambda_k}$ is an analytic increasing function of $p\in]1/k, 1[$.

Now, observe that $\rho\longmapsto \int_0^\infty \log x d\nu_{k,\rho}(x)$ is continuous on $[0,1]$ (as the uniform limit of continuous functions).
When $p$ goes to zero in the linear case (or $p\to 1/k$ in the non-linear case), $p_R$ tends to 0 and $\rho$ tends to 1. By continuity of the integral, we obtain~\eqref{limit} using Remark~\ref{rho = 1}. 
When $p=1$, the deterministic sequence $F_n=\widetilde F_n$ grows exponentially fast, and the expression of $\gamma_{1,\lambda_k}$ follows from elementary analysis.

When $\lambda \ge 2$ (we do not need to distinguish the linear case from the non-linear cases), the proof is handled in the same way, using Proposition~\ref{compare}.
\end{proof}

\begin{proof}[Proof of Proposition~\ref{compare} when $\lambda\ge2$]
We let the reader check that in this case, for all $s\ge0$ the matrix $RL^s$ has nonnegative entries. 
Suppose the difference between $(X_i)$ and $(X'_i)$ occurs at level $j$. For any $n\ge j$, the sequence $X_j\ldots X_n$ can be decomposed into blocks of the form $RL^s$, $s\ge0$, hence the product of matrices $X_j\cdots X_n$ has nonnegative entries. 
If $n\ge j$, we can thus write $F'_n$ as a linear combination with nonnegative coefficients: $ F'_n = C_1 F'_{j-2}+ C_2 F'_{j-1}$.
Moreover, $F_n = -C_1 F_{j-2}+ C_2 F_{j-1}=-C_1 F'_{j-2}+ C_2 F'_{j-1}$, hence $F_n\le F'_n$ (since $F_2/F_1\ge 1$, all $F_n$'s are positive).
\end{proof}

The proof of Proposition ~\ref{compare} when $\lambda=\lambda_k$ uses three lemmas. The first one can be viewed as a particular case when the sequence of $R$'s and $L$'s is reduced.

\begin{lemma}
\label{compareRL}
Let $\lambda=\lambda_k$.
Let $a>0$, $b>0$, $j_1\ge0$ and $j_2\ge0$ such that $j_1+1+j_2\le k-2$.
If $(a',b')=(a,b)RL^{j_1}RL^{j_2}$ and $(a'',b'')=(a,b)RL^{j_1+1+j_2}$, then $b'\ge b''$.
\end{lemma}
\begin{proof}
For any $\ell\in\{0,\ldots,j_2\}$, set $(x_\ell,x_{\ell+1})\egdef (a,b)RL^{j_1}RL^{\ell}$, and $(y_\ell,y_{\ell+1})\egdef (a,b)RL^{j_1+1+\ell}$. Then the quotient $x_{\ell+1}/x_\ell$ lies in $I_\ell$ (see Section~\ref{Sec:SternBrocot}), whereas the quotient $y_{\ell+1}/y_\ell$ lies in $I_{j_1+1+\ell}$. It follows that $y_{\ell+1}/y_\ell \le x_{\ell+1}/x_\ell$, and since $x_0=y_0$, we inductively get that for all $\ell\in\{0,\ldots,j_2+1\}$, $y_\ell\le x_\ell$. The lemma is proved, observing that $b'=x_{j_2+1}$ and $b''=y_{j_2+1}$. 
\end{proof}

\begin{lemma}
\label{comparek}
Let $\lambda=\lambda_k$.
Let $(X_i)_{i\ge2}$ be a sequence of matrices in $\{R,L\}$, which does not contain $k-1$ consecutive $L$'s and such that $X_2=R$. Let $x_0>0$, $x_1>0$, and set inductively $(x_i,x_{i+1})\egdef (x_{i-1},x_i)X_{i+1}$. Then for any $i\ge0$, $x_{i+k}\ge x_i$.
\end{lemma}

\begin{proof}
If $X_{i+1}=R$, this is just a repeated application of the following claim: If $a>0$, $b>0$, $0\le j\le k-3$, and if we set $(a',b')\egdef (a,b)RL^j$, then $b'\ge b$. Indeed, by \eqref{RtimesPowersOfL}, we have $b'\ge b\, \sin\bigl((j+2)\pi/k\bigr) / \sin\bigl(\pi/k\bigr) \ge b$.

If $X_{i+1}=L$, we first prove the lemma when the sequence $X_{i+1}\ldots X_{i+k}$ contains only one $R$: $X_{i+j}=R$ for some $j\in\{2,\ldots,k-1\}$. We proceed by induction on $j$. 
If $j=2$, then $ (x_{i+k-1},x_{i+k}) = (x_{i-1},x_{i}) LRL^{k-2}$.
By \eqref{RtimesPowersOfL}, the second column of $RL^{k-2}$ is $\begin{pmatrix}1\\0\end{pmatrix}$, thus $x_{i+k}=x_i$. 
Now, assume $j>2$ and that we have proved the inequality up to $j-1$. 
Since the sequence of matrices starts with an $R$ and does not contain $k-1$ consecutive $L$'s, we have $x_{i+1}/x_i \in I_\ell$ for some $\ell\le k-j$ (see Section~\ref{Sec:SternBrocot}). 
In particular, $x_{i+1}/x_i \ge b_{k-j}$. 
Now define $x'_{i+k+1}$ by $(x_{i+k},x'_{i+k+1})\egdef (x_{i+k-1},x_{i+k}) L$. 
We have $x'_{i+k+1}/x_{i+k} \in I_{k-j+1}$, thus is bounded below by $b_{k-j}$. 
Using the induction hypothesis $x'_{i+k+1}\ge x_{i+1}$, we conclude that $x_{i+k}\ge x_i$. 

Finally, assume that the sequence $X_{i+1}\ldots X_{i+k}$ starts with an $L$ and contains several $R$'s. Turning the last $R$ into an $L$, we can apply Lemma~\ref{compareRL} to compare $x_{i+k}$ with the case where there is one less $R$, and prove the result by induction on the number of $R$'s.
\end{proof}

\begin{lemma}
 \label{reverse}
Let $\lambda=\lambda_k$.
Let $\widetilde F_{n}$ be inductively defined by $\widetilde F_0\ge 0$, $\widetilde F_1\ge 0$ and 
$\widetilde F_{n+1}= |\lambda \widetilde F_{n}-\widetilde F_{n-1}|$ for any $n\ge 1$.
Then for any $n\ge 0$, $\widetilde F_{n+k}\le \widetilde F_{n}$.
\end{lemma}

\begin{proof}
For $n\le 0$, let $G_{n} \egdef \widetilde F_{-n}\ge 0$. 
Then, for any $n\le -1$, we have 
$$
(G_{n}, G_{n+1})= \begin{cases}
                        (G_{n-1}, G_n)L & \mbox{ if }\lambda \widetilde F_{n}\ge\widetilde F_{n-1},\\
			(G_{n-1}, G_n)R & \mbox{ otherwise.}
                   \end{cases}
$$
Moreover, we can assume that the sequence of matrices in $\{R,L\}$ corresponding to $(G_n)$ never contains $k-1$ consecutive $L$'s. 
Indeed, the second column of $L^{k-1}$ is $\begin{pmatrix}-1\\0\end{pmatrix}$. 
Thus, if we had $k-1$ consecutive $L$'s, we could find $n$ such that $-G_{n-1}=G_{n+k-1}$, which is possible only if $G_{n-1}=G_{n+k-1}=0$.
But if such a situation occurs we can always turn the first $L$ into an $R$ without changing the sequence (because $(0, G_n)R=(0, G_n)L$). The result is thus a direct application of Lemma~\ref{comparek}.
\end{proof}

\begin{proof}[Proof of Proposition~\ref{compare} when $\lambda=\lambda_k$]
Suppose the difference between $(X_i)$ and $(X'_i)$ occurs at level $j$.
We decompose $(X_{j})_{i\ge j}$ as $LL^rY$ and $(X'_{j})_{i\ge j}$ as $RL^rY$, where $0\le r\le +\infty$ and $Y=(Y_i)_{i\ge j+r+1}$ is a sequence of letters in the alphabet $\{R, L\}$ such that $Y_{j+r+1}=R$. 

\medskip
Suppose first that, after the difference, all letters  are $L$'s ($Y=\emptyset$). 
Let $j_1\in\{0, \dots, k-2\}$ be such that $\widetilde F_{j-1}/\widetilde F_j \in I_{j_1}$. 
Without loss of generality, we can assume that the sequences $(X_i)$ and $(X'_i)$ are reduced before their first difference. 
Then, $X_{j-j_1-1}\dots X_{j-1}=X'_{j-j_1-1}\dots X'_{j-1}=RL^{j_1}$. 

By Lemma~\ref{compareRL}, $\widetilde F_{j+s}\ge \widetilde F'_{j+s}$ for all $0\le s\le j_2$, where $j_2:= k-3-j_1$. 

Now, by Lemma~\ref{comparek}, for all $1+j_2\le s\le k-2$, 
$\widetilde F_{j+s}\ge \widetilde F_{j+s-k}$, which is equal to $\widetilde F'_{j+s-k}$ since $s<k$. 
On the other hand, when $s=j_2+1$, we have $\widetilde F'_{j+j_2+1-k}=\widetilde F'_{j+j_2+1}$ because $X'_{j-j_1-1}\dots X'_{j+j_2+1}=RL^{k-1}$. Moreover, by Lemma~\ref{reverse}, $\widetilde F'_{j+s-k}\ge \widetilde F'_{j+s}$ for all $1+j_2< s\le k-2$. 
We thus get that $\widetilde F_{j+s}\ge \widetilde F'_{j+s}$ for all $j_2+1\le s\le k-2$.

If $s\ge k-1$, reducing the pattern $RL^{k-1}$ in the sequence $(X_{j})_{i\ge j}$, we have $\widetilde F_{j+s}=\widetilde F'_{j+s-k}$ which is larger than $\widetilde F'_{j+s}$ by Lemma~\ref{reverse}.

\medskip
Suppose now that the suffix $Y$ is reduced. 
The above argument shows that all labels up to $j+r$ are well-ordered: 
In particular, $\widetilde F_{j+r-1} \le \widetilde F'_{j+r-1}$ and $\widetilde F_{j+r} \le \widetilde F'_{j+r}$.
Since $Y$ is reduced, we can write, for any $n\ge j+r$,
$(\widetilde F_{n}, \widetilde F_{n+1}) = (\widetilde F_{j+r-1}, \widetilde F_{j+r})Y_{j+r+1}\cdots Y_{n+1}$, where each $Y_i$ is interpreted as the corresponding matrix (the same equality is valid if we replace $\widetilde F$ by $\widetilde F'$). 
The product $Y_{j+r+1}\cdots Y_{n+1}$ can be decomposed into blocks of the form $RL^\ell$, with $0\le \ell\le k-2$, which are matrices with nonnegative entries. 
Therefore, for any $n\ge j+r$, the label $\widetilde F_{n}$ is a linear combination of $\widetilde F_{j+r-1}$ and $\widetilde F_{j+r}$ , with nonnegative coefficients. 
Moreover, it is also true with the same coefficients if we replace $\widetilde F$ by $\widetilde F'$.
We conclude that $\widetilde F_{n}\le \widetilde F'_{n}$.

In the general case, we make all possible reductions on $Y$. We are left either with a reduced sequence or with a sequence of $L$'s, which are the two situations we have already studied.
\end{proof}

\begin{remark}
In~\cite{janvresse2007}, a formula for the derivative of $\gamma_{p,1}$ with respect to $p$ was given, involving the product measure $\nu_{3,\rho}\otimes\nu_{3,\rho}$. We do not know whether this formula can be generalized to other $k$'s.

\end{remark}

\subsection{Variations with $\lambda$}
\label{Sec:variations_k}

For $p=1$, the deterministic sequence $F_n=\widetilde F_n$ grows exponentially fast, and we have in that case
$$\widetilde  \gamma_{1,\lambda} =\gamma_{1,\lambda} =  \log \left( \dfrac{\lambda+\sqrt{\lambda^2+4}}{2} \right), $$
which is increasing with $\lambda$.

We conjecture that, when $p$ is fixed, $\gamma_{p,\lambda_k}$ and $\widetilde  \gamma_{p,\lambda_k}$ are increasing with $k$, and that $\gamma_{p,\lambda}$ is increasing with $\lambda$ for $\lambda\ge2$ (see Figure~\ref{Fig:gammas}).

\begin{figure}[h]
	\begin{center}
	\input{gammas.pstex_t}
	\end{center}
\caption{The value of $\gamma_{p,\lambda}$ (linear case, left) and $\widetilde\gamma_{p,\lambda}$ (non-linear case, right) for $\lambda=\lambda_k$, $k=3, 4, 5, 10$, $\lambda=2$ (bold), $\lambda=2.05$, $2.1$, $2.5$ and $3$. Numerical computations support the conjecture that $\gamma_{p,\lambda}$ and $\widetilde\gamma_{p,\lambda}$ are increasing with $\lambda$.}
\label{Fig:gammas}
\end{figure}
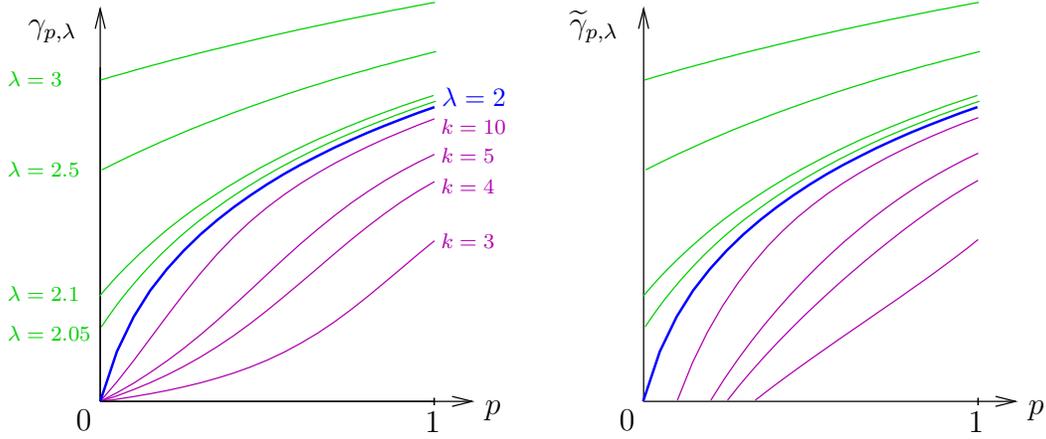 

\section{Connections with Embree-Trefethen's paper}
\label{Sec:Embree-Trefethen}
\subsection{Positivity of the Lyapunov exponent}
We have proved that the largest Lyapunov exponent corresponding to the linear $\lambda$-random Fibonacci sequence is positive for all $p$.
In~\cite{embree1999}, Embree and Trefethen study a slight modification of our linear random Fibonacci sequence when $p=1/2$. To be exact, they study the random sequence $x_{n+1}=x_n\pm \beta x_{n-1}$, which by a simple rescaling gives our linear $\lambda$-random Fibonacci sequence where $\lambda=1/\sqrt{\beta}$ (see our introduction). However, the exponential growth is not preserved by this rescaling.  More precisely, the exponential growth $\sigma(\beta)=\lim |x_n|^{1/n}$ of Embree and Trefethen's sequence satisfies
$$ \log \sigma(\beta) = \gamma_{1/2,\lambda} - \log\lambda. $$
In particular, $\sigma(\beta)<1$ if and only if $\gamma_{1/2,\lambda} < \log\lambda$, which according to the simulations described in their paper happens for $\beta<\beta^*\approx 0.70258\ldots$ (which corresponds to $\lambda>1.19303\ldots$).

By Theorem~\ref{croissance}, the function $p\mapsto\gamma_{p,\lambda}$ is continuous and increasing from $0$ to $\gamma_{1,\lambda}>\log\lambda$. Hence there exists a unique $p^*(\lambda)\in[0,1]$ such that, for $p<p^*$, $\gamma_{p,\lambda}<\log\lambda$ and for $p>p^*$, $\gamma_{p,\lambda}>\log\lambda$.
According to~\cite{embree1999}, for $\lambda=1$ we have $p^*<1/2$, and for $\lambda=\lambda_k$ ($k\ge 4$) and $\lambda\ge2$, $p^*>1/2$.

For $\lambda\ge2$, we can indeed prove that $\gamma_{1/2,\lambda}<\log \lambda$: By Jensen's inequality, we have
$$ \gamma_{1/2,\lambda} < \log \left(\int_B^{\lambda+1/B} x \, d\mu_{1/2,\lambda}\right), $$
which is equal to $\log\lambda$ by symmetry of the measure $\mu_{1/2,\lambda}$.

For $\lambda=1$, we know that $\gamma_{p,1}>0$ for all $p>0$ thus $p^*=0$. When $\lambda=\lambda_k$, $k\ge 4$, numerical computations of the integral confirm that $p^*>1/2$, but we do not know how to prove it.

\subsection{Sign-flip frequency}
Embree and Trefethen  introduce the \textit{sign-flip frequency} as the proportion of values $n$ such that $F_nF_{n+1}<0$, and give (without proof) the estimate $2^{-\pi\lambda/\sqrt{4-\lambda^2}}$ for this frequency, as $\lambda\to2$, $\lambda<2$.

Note that, for $\lambda\ge 2$, there are no sign change as soon as $n$ is large enough, and the sign-flip frequency is zero.

For $\lambda=\lambda_k$, recall that for $n$ large enough, the sign of the reduced sequence $(F_n^r)$ is constant (see Lemma~\ref{L+}). 
Moreover, by~\eqref{reduction} and the fact that for all $0\le j\le k-2$ the matrix $RL^j$ has nonnegative entries (see~\eqref{RtimesPowersOfL}), the product $F_nF_{n}^r$ changes sign if and only if a pattern $RL^{k-1}$ is removed.
Thus, the sign-flip frequency is equal to the frequency of deletions in the reduction process.

Note that we have to make sure that this frequency indeed exists. 
This can be seen by considering the reduction of the left-infinite i.i.d. sequence $(X^*)_{-\infty}^0$ (Section~\ref{reductioninfinie}), since for $n$ large enough, deletions in the reduction process of $(X)_3^n$ occur at the same times as in the reduction process of $(X^*)_{-\infty}^n$. 
In the latter case, the ergodic theorem ensures that the frequency $\sigma$ of deletions exists and is equal to the probability that $(X^*)_{-\infty}^0$ be not proper.
By Lemma~\ref{Lem:excursion}, $(X^*)_{-\infty}^0$ is not proper if and only if there exists a unique $\ell>0$ such that $(X^*)_{-\ell}^0$ is an excursion, and $(X^*)_{-\infty}^{-\ell-1}$ is proper. Thus,
$$\sigma = \sum_{w \mbox{\begin{scriptsize} excursions\end{scriptsize}}}\ \PP(w) (1-\sigma).$$
By~\eqref{eq:excursion}, we get that the sign-flip frequency is equal to
\begin{equation}
 \sigma = \sigma(\lambda_k,p) = \frac{p(1-p_R)}{p+(1-p)p_R+p(1-p_R)}. 
\end{equation}

Now, for a fixed $p\in ]0,1[$, we would like to obtain an estimate for $\sigma$ as $k\to\infty$. 
First, observe that $p_R=p_R(k)\to 1$ as $k\to\infty$. Indeed, recalling the expression of the function~$g$ given by~\eqref{survival-pr}, for any $x\in ]0,1[$, we have $g(x)<0$ for $k$ large enough, which implies $p_R>x$. Then, since $p_R$ satisfies
$$ 1-p_R = \left(1-\dfrac{pp_R}{p+(1-p)p_R}\right)^{k-1}, $$
we get that $p_R\to 1$ exponentially fast with $k$. Using this estimation in the above equation, elementary computations lead to 
$$1-p_R \mathop{\sim}_{k\to\infty} (1-p)^{k-1}.$$
Thus, 
$$\sigma(\lambda_k,p)\mathop{\sim}_{k\to\infty} p(1-p)^{k-1}.$$

For $p=1/2$, this proves the estimate provided in~\cite{embree1999} in the special case $\lambda=\lambda_k$.

\bibliography{rf-rosen.bib}
\end{document}

%% file: mesure4.pstex_t
\begin{picture}(0,0)%
\includegraphics{mesure4.pstex}%
\end{picture}%
\setlength{\unitlength}{1973sp}%
\begingroup\makeatletter\ifx\SetFigFont\undefined%
\gdef\SetFigFont#1#2#3#4#5{%
  \reset@font\fontsize{#1}{#2pt}%
  \fontfamily{#3}\fontseries{#4}\fontshape{#5}%
  \selectfont}%
\fi\endgroup%
\begin{picture}(11512,4882)(511,-3827)
\put(4501,764){\makebox(0,0)[lb]{\smash{{\SetFigFont{11}{13.2}{\rmdefault}{\mddefault}{\updefault}{\color[rgb]{0,0,0}$1/\sqrt{2}$}%
}}}}
\put(8176,764){\makebox(0,0)[lb]{\smash{{\SetFigFont{11}{13.2}{\rmdefault}{\mddefault}{\updefault}{\color[rgb]{0,0,0}$\sqrt{2}$}%
}}}}
\put(11776,764){\makebox(0,0)[lb]{\smash{{\SetFigFont{11}{13.2}{\rmdefault}{\mddefault}{\updefault}{\color[rgb]{0,0,0}$\infty$}%
}}}}
\put(1051,764){\makebox(0,0)[lb]{\smash{{\SetFigFont{11}{13.2}{\rmdefault}{\mddefault}{\updefault}{\color[rgb]{0,0,0}$0$}%
}}}}
\put(1426,-1411){\makebox(0,0)[lb]{\smash{{\SetFigFont{9}{10.8}{\rmdefault}{\mddefault}{\updefault}{\color[rgb]{0,0,.82}$I_{2,0}$}%
}}}}
\put(2626,-1411){\makebox(0,0)[lb]{\smash{{\SetFigFont{9}{10.8}{\rmdefault}{\mddefault}{\updefault}{\color[rgb]{0,0,.82}$I_{2,1}$}%
}}}}
\put(3826,-1411){\makebox(0,0)[lb]{\smash{{\SetFigFont{9}{10.8}{\rmdefault}{\mddefault}{\updefault}{\color[rgb]{0,0,.82}$I_{2,2}$}%
}}}}
\put(1801,464){\makebox(0,0)[lb]{\smash{{\SetFigFont{6}{7.2}{\rmdefault}{\mddefault}{\updefault}{\color[rgb]{0,0,0}$[1,-1,1,1]_{\sqrt{2}}$}%
}}}}
\put(5401,464){\makebox(0,0)[lb]{\smash{{\SetFigFont{6}{7.2}{\rmdefault}{\mddefault}{\updefault}{\color[rgb]{0,0,0}$[1,-1,-1]_{\sqrt{2}}$}%
}}}}
\put(9376,464){\makebox(0,0)[lb]{\smash{{\SetFigFont{6}{7.2}{\rmdefault}{\mddefault}{\updefault}{\color[rgb]{0,0,0}$[1,1]_{\sqrt{2}}$}%
}}}}
\put(2851,-61){\makebox(0,0)[lb]{\smash{{\SetFigFont{6}{7.2}{\rmdefault}{\mddefault}{\updefault}{\color[rgb]{0,0,0}$[1,-1,1,1,-1]_{\sqrt{2}}$}%
}}}}
\put(6451,-61){\makebox(0,0)[lb]{\smash{{\SetFigFont{6}{7.2}{\rmdefault}{\mddefault}{\updefault}{\color[rgb]{0,0,0}$[1,-1,-1,1]_{\sqrt{2}}$}%
}}}}
\put(10351,-61){\makebox(0,0)[lb]{\smash{{\SetFigFont{6}{7.2}{\rmdefault}{\mddefault}{\updefault}{\color[rgb]{0,0,0}$[1,1,-1]_{\sqrt{2}}$}%
}}}}
\put(526,-436){\makebox(0,0)[lb]{\smash{{\SetFigFont{9}{10.8}{\rmdefault}{\mddefault}{\updefault}{\color[rgb]{0,0,0}$[1,-1,1]_{\sqrt{2}}$}%
}}}}
\put(4426,-436){\makebox(0,0)[lb]{\smash{{\SetFigFont{9}{10.8}{\rmdefault}{\mddefault}{\updefault}{\color[rgb]{0,0,0}$[1,-1]_{\sqrt{2}}$}%
}}}}
\put(8176,-436){\makebox(0,0)[lb]{\smash{{\SetFigFont{9}{10.8}{\rmdefault}{\mddefault}{\updefault}{\color[rgb]{0,0,0}$[1]_{\sqrt{2}}$}%
}}}}
\put(6226,-3736){\makebox(0,0)[lb]{\smash{{\SetFigFont{9}{10.8}{\rmdefault}{\mddefault}{\updefault}{\color[rgb]{0,0,0}$\rho/Z$}%
}}}}
\put(3976,-1936){\makebox(0,0)[lb]{\smash{{\SetFigFont{9}{10.8}{\rmdefault}{\mddefault}{\updefault}{\color[rgb]{0,0,0}$\frac{\rho^4}{Z^2}$}%
}}}}
\put(5176,-1936){\makebox(0,0)[lb]{\smash{{\SetFigFont{9}{10.8}{\rmdefault}{\mddefault}{\updefault}{\color[rgb]{0,0,0}$\frac{\rho}{Z^2}$}%
}}}}
\put(6376,-1936){\makebox(0,0)[lb]{\smash{{\SetFigFont{9}{10.8}{\rmdefault}{\mddefault}{\updefault}{\color[rgb]{0,0,0}$\frac{\rho^2}{Z^2}$}%
}}}}
\put(7501,-1936){\makebox(0,0)[lb]{\smash{{\SetFigFont{9}{10.8}{\rmdefault}{\mddefault}{\updefault}{\color[rgb]{0,0,0}$\frac{\rho^3}{Z^2}$}%
}}}}
\put(8701,-1936){\makebox(0,0)[lb]{\smash{{\SetFigFont{9}{10.8}{\rmdefault}{\mddefault}{\updefault}{\color[rgb]{0,0,0}$\frac{1}{Z^2}$}%
}}}}
\put(9976,-1936){\makebox(0,0)[lb]{\smash{{\SetFigFont{9}{10.8}{\rmdefault}{\mddefault}{\updefault}{\color[rgb]{0,0,0}$\frac{\rho}{Z^2}$}%
}}}}
\put(10051,-3136){\makebox(0,0)[lb]{\smash{{\SetFigFont{9}{10.8}{\rmdefault}{\mddefault}{\updefault}{\color[rgb]{0,0,.82}$I_0$}%
}}}}
\put(5101,-1411){\makebox(0,0)[lb]{\smash{{\SetFigFont{9}{10.8}{\rmdefault}{\mddefault}{\updefault}{\color[rgb]{0,0,.82}$I_{1,0}$}%
}}}}
\put(6301,-1411){\makebox(0,0)[lb]{\smash{{\SetFigFont{9}{10.8}{\rmdefault}{\mddefault}{\updefault}{\color[rgb]{0,0,.82}$I_{1,1}$}%
}}}}
\put(7501,-1411){\makebox(0,0)[lb]{\smash{{\SetFigFont{9}{10.8}{\rmdefault}{\mddefault}{\updefault}{\color[rgb]{0,0,.82}$I_{1,2}$}%
}}}}
\put(8701,-1411){\makebox(0,0)[lb]{\smash{{\SetFigFont{9}{10.8}{\rmdefault}{\mddefault}{\updefault}{\color[rgb]{0,0,.82}$I_{0,0}$}%
}}}}
\put(9901,-1411){\makebox(0,0)[lb]{\smash{{\SetFigFont{9}{10.8}{\rmdefault}{\mddefault}{\updefault}{\color[rgb]{0,0,.82}$I_{0,1}$}%
}}}}
\put(11101,-1411){\makebox(0,0)[lb]{\smash{{\SetFigFont{9}{10.8}{\rmdefault}{\mddefault}{\updefault}{\color[rgb]{0,0,.82}$I_{0,2}$}%
}}}}
\put(11101,-1936){\makebox(0,0)[lb]{\smash{{\SetFigFont{9}{10.8}{\rmdefault}{\mddefault}{\updefault}{\color[rgb]{0,0,0}$\frac{\rho^2}{Z^2}$}%
}}}}
\put(2701,-3136){\makebox(0,0)[lb]{\smash{{\SetFigFont{9}{10.8}{\rmdefault}{\mddefault}{\updefault}{\color[rgb]{0,0,.82}$I_2$}%
}}}}
\put(6376,-3136){\makebox(0,0)[lb]{\smash{{\SetFigFont{9}{10.8}{\rmdefault}{\mddefault}{\updefault}{\color[rgb]{0,0,.82}$I_1$}%
}}}}
\put(2401,-3736){\makebox(0,0)[lb]{\smash{{\SetFigFont{9}{10.8}{\rmdefault}{\mddefault}{\updefault}{\color[rgb]{0,0,0}$\rho^2/Z$}%
}}}}
\put(9976,-3736){\makebox(0,0)[lb]{\smash{{\SetFigFont{9}{10.8}{\rmdefault}{\mddefault}{\updefault}{\color[rgb]{0,0,0}$1/Z$}%
}}}}
\put(2701,-1936){\makebox(0,0)[lb]{\smash{{\SetFigFont{9}{10.8}{\rmdefault}{\mddefault}{\updefault}{\color[rgb]{0,0,0}$\frac{\rho^3}{Z^2}$}%
}}}}
\put(1426,-1936){\makebox(0,0)[lb]{\smash{{\SetFigFont{9}{10.8}{\rmdefault}{\mddefault}{\updefault}{\color[rgb]{0,0,0}$\frac{\rho^2}{Z^2}$}%
}}}}
\end{picture}%

%% file: cercle.pstex_t
\begin{picture}(0,0)%
\includegraphics{cercle.pstex}%
\end{picture}%
\setlength{\unitlength}{4144sp}%
\begingroup\makeatletter\ifx\SetFigFont\undefined%
\gdef\SetFigFont#1#2#3#4#5{%
  \reset@font\fontsize{#1}{#2pt}%
  \fontfamily{#3}\fontseries{#4}\fontshape{#5}%
  \selectfont}%
\fi\endgroup%
\begin{picture}(3935,4429)(-1807,-1439)
\put(1981,344){\makebox(0,0)[lb]{\smash{{\SetFigFont{12}{14.4}{\rmdefault}{\mddefault}{\updefault}{\color[rgb]{0,0,0}$\pi/k$}%
}}}}
\put(631,1874){\makebox(0,0)[lb]{\smash{{\SetFigFont{12}{14.4}{\rmdefault}{\mddefault}{\updefault}{\color[rgb]{0,0,0}$\pi/k$}%
}}}}
\put(1711,1964){\makebox(0,0)[lb]{\smash{{\SetFigFont{12}{14.4}{\rmdefault}{\mddefault}{\updefault}{\color[rgb]{0,0,0}$\pi/k$}%
}}}}
\put(451,2819){\makebox(0,0)[lb]{\smash{{\SetFigFont{12}{14.4}{\rmdefault}{\mddefault}{\updefault}{\color[rgb]{0,0,0}$\pi/k$}%
}}}}
\put(1756,-1366){\makebox(0,0)[lb]{\smash{{\SetFigFont{12}{14.4}{\rmdefault}{\mddefault}{\updefault}{\color[rgb]{0,0,0}$\widetilde F_{n}$}%
}}}}
\put(252,-1366){\makebox(0,0)[lb]{\smash{{\SetFigFont{12}{14.4}{\rmdefault}{\mddefault}{\updefault}{\color[rgb]{0,0,0}$\widetilde F_{n+2}$}%
}}}}
\put(866,-1366){\makebox(0,0)[lb]{\smash{{\SetFigFont{12}{14.4}{\rmdefault}{\mddefault}{\updefault}{\color[rgb]{0,0,0}$\widetilde F_{n+1}$}%
}}}}
\put(1309,-1366){\makebox(0,0)[lb]{\smash{{\SetFigFont{12}{14.4}{\rmdefault}{\mddefault}{\updefault}{\color[rgb]{0,0,0}$\widetilde F_{n-1}$}%
}}}}
\put(-224,704){\makebox(0,0)[lb]{\smash{{\SetFigFont{12}{14.4}{\rmdefault}{\mddefault}{\updefault}{\color[rgb]{0,0,0}$O$}%
}}}}
\end{picture}%

%% file: mesure2.pstex_t
\begin{picture}(0,0)%
\includegraphics{mesure2.pstex}%
\end{picture}%
\setlength{\unitlength}{2368sp}%
\begingroup\makeatletter\ifx\SetFigFont\undefined%
\gdef\SetFigFont#1#2#3#4#5{%
  \reset@font\fontsize{#1}{#2pt}%
  \fontfamily{#3}\fontseries{#4}\fontshape{#5}%
  \selectfont}%
\fi\endgroup%
\begin{picture}(10921,4825)(1111,-11081)
\put(1126,-6511){\makebox(0,0)[lb]{\smash{{\SetFigFont{12}{14.4}{\rmdefault}{\mddefault}{\updefault}{\color[rgb]{0,0,0}$1$}%
}}}}
\put(11926,-6511){\makebox(0,0)[lb]{\smash{{\SetFigFont{12}{14.4}{\rmdefault}{\mddefault}{\updefault}{\color[rgb]{0,0,0}$3$}%
}}}}
\put(2131,-6781){\makebox(0,0)[lb]{\smash{{\SetFigFont{7}{8.4}{\rmdefault}{\mddefault}{\updefault}{\color[rgb]{0,0,0}$4/3$}%
}}}}
\put(3076,-6781){\makebox(0,0)[lb]{\smash{{\SetFigFont{7}{8.4}{\rmdefault}{\mddefault}{\updefault}{\color[rgb]{0,0,0}$7/5$}%
}}}}
\put(3751,-6781){\makebox(0,0)[lb]{\smash{{\SetFigFont{7}{8.4}{\rmdefault}{\mddefault}{\updefault}{\color[rgb]{0,0,0}$11/7$}%
}}}}
\put(4741,-6781){\makebox(0,0)[lb]{\smash{{\SetFigFont{7}{8.4}{\rmdefault}{\mddefault}{\updefault}{\color[rgb]{0,0,0}$8/5$}%
}}}}
\put(8131,-6781){\makebox(0,0)[lb]{\smash{{\SetFigFont{7}{8.4}{\rmdefault}{\mddefault}{\updefault}{\color[rgb]{0,0,0}$12/5$}%
}}}}
\put(9121,-6781){\makebox(0,0)[lb]{\smash{{\SetFigFont{7}{8.4}{\rmdefault}{\mddefault}{\updefault}{\color[rgb]{0,0,0}$17/7$}%
}}}}
\put(10801,-6781){\makebox(0,0)[lb]{\smash{{\SetFigFont{7}{8.4}{\rmdefault}{\mddefault}{\updefault}{\color[rgb]{0,0,0}$8/3$}%
}}}}
\put(9811,-6781){\makebox(0,0)[lb]{\smash{{\SetFigFont{7}{8.4}{\rmdefault}{\mddefault}{\updefault}{\color[rgb]{0,0,0}$13/5$}%
}}}}
\put(3286,-10471){\makebox(0,0)[lb]{\smash{{\SetFigFont{11}{13.2}{\rmdefault}{\mddefault}{\updefault}{\color[rgb]{0,0,.82}$I_L$}%
}}}}
\put(9106,-10471){\makebox(0,0)[lb]{\smash{{\SetFigFont{11}{13.2}{\rmdefault}{\mddefault}{\updefault}{\color[rgb]{0,0,.82}$I_R$}%
}}}}
\put(3376,-10981){\makebox(0,0)[lb]{\smash{{\SetFigFont{11}{13.2}{\rmdefault}{\mddefault}{\updefault}{\color[rgb]{0,0,0}$1-p$}%
}}}}
\put(9181,-10981){\makebox(0,0)[lb]{\smash{{\SetFigFont{11}{13.2}{\rmdefault}{\mddefault}{\updefault}{\color[rgb]{0,0,0}$p$}%
}}}}
\put(2011,-9226){\makebox(0,0)[lb]{\smash{{\SetFigFont{11}{13.2}{\rmdefault}{\mddefault}{\updefault}{\color[rgb]{0,0,.82}$I_{LL}$}%
}}}}
\put(4516,-9226){\makebox(0,0)[lb]{\smash{{\SetFigFont{11}{13.2}{\rmdefault}{\mddefault}{\updefault}{\color[rgb]{0,0,.82}$I_{RL}$}%
}}}}
\put(7981,-9226){\makebox(0,0)[lb]{\smash{{\SetFigFont{11}{13.2}{\rmdefault}{\mddefault}{\updefault}{\color[rgb]{0,0,.82}$I_{RR}$}%
}}}}
\put(10276,-9226){\makebox(0,0)[lb]{\smash{{\SetFigFont{11}{13.2}{\rmdefault}{\mddefault}{\updefault}{\color[rgb]{0,0,.82}$I_{LR}$}%
}}}}
\put(1441,-8011){\makebox(0,0)[lb]{\smash{{\SetFigFont{7}{8.4}{\rmdefault}{\mddefault}{\updefault}{\color[rgb]{0,0,.82}$I_{LLL}$}%
}}}}
\put(7411,-8011){\makebox(0,0)[lb]{\smash{{\SetFigFont{7}{8.4}{\rmdefault}{\mddefault}{\updefault}{\color[rgb]{0,0,.82}$I_{LRR}$}%
}}}}
\put(8521,-8011){\makebox(0,0)[lb]{\smash{{\SetFigFont{7}{8.4}{\rmdefault}{\mddefault}{\updefault}{\color[rgb]{0,0,.82}$I_{RRR}$}%
}}}}
\put(10096,-8011){\makebox(0,0)[lb]{\smash{{\SetFigFont{7}{8.4}{\rmdefault}{\mddefault}{\updefault}{\color[rgb]{0,0,.82}$I_{RLR}$}%
}}}}
\put(2011,-9691){\makebox(0,0)[lb]{\smash{{\SetFigFont{11}{13.2}{\rmdefault}{\mddefault}{\updefault}{\color[rgb]{0,0,0}$(1-p)^2$}%
}}}}
\put(4441,-9691){\makebox(0,0)[lb]{\smash{{\SetFigFont{11}{13.2}{\rmdefault}{\mddefault}{\updefault}{\color[rgb]{0,0,0}$(1-p)p$}%
}}}}
\put(10381,-9691){\makebox(0,0)[lb]{\smash{{\SetFigFont{11}{13.2}{\rmdefault}{\mddefault}{\updefault}{\color[rgb]{0,0,0}$(1-p)p$}%
}}}}
\put(8116,-9691){\makebox(0,0)[lb]{\smash{{\SetFigFont{11}{13.2}{\rmdefault}{\mddefault}{\updefault}{\color[rgb]{0,0,0}$p^2$}%
}}}}
\put(7381,-8491){\makebox(0,0)[lb]{\smash{{\SetFigFont{7}{8.4}{\rmdefault}{\mddefault}{\updefault}{\color[rgb]{0,0,0}$(1-p)p^2$}%
}}}}
\put(3496,-6556){\makebox(0,0)[lb]{\smash{{\SetFigFont{11}{13.2}{\rmdefault}{\mddefault}{\updefault}{\color[rgb]{0,0,0}$3/2$}%
}}}}
\put(5821,-6586){\makebox(0,0)[lb]{\smash{{\SetFigFont{11}{13.2}{\rmdefault}{\mddefault}{\updefault}{\color[rgb]{0,0,0}$5/3$}%
}}}}
\put(7081,-6571){\makebox(0,0)[lb]{\smash{{\SetFigFont{11}{13.2}{\rmdefault}{\mddefault}{\updefault}{\color[rgb]{0,0,0}$7/3$}%
}}}}
\put(9451,-6571){\makebox(0,0)[lb]{\smash{{\SetFigFont{11}{13.2}{\rmdefault}{\mddefault}{\updefault}{\color[rgb]{0,0,0}$5/2$}%
}}}}
\put(8611,-8491){\makebox(0,0)[lb]{\smash{{\SetFigFont{7}{8.4}{\rmdefault}{\mddefault}{\updefault}{\color[rgb]{0,0,0}$p^3$}%
}}}}
\put(9961,-8491){\makebox(0,0)[lb]{\smash{{\SetFigFont{7}{8.4}{\rmdefault}{\mddefault}{\updefault}{\color[rgb]{0,0,0}$(1-p)p^2$}%
}}}}
\put(11026,-8491){\makebox(0,0)[lb]{\smash{{\SetFigFont{7}{8.4}{\rmdefault}{\mddefault}{\updefault}{\color[rgb]{0,0,0}$(1-p)^2p$}%
}}}}
\put(3871,-8491){\makebox(0,0)[lb]{\smash{{\SetFigFont{7}{8.4}{\rmdefault}{\mddefault}{\updefault}{\color[rgb]{0,0,0}$(1-p)p^2$}%
}}}}
\put(5056,-8491){\makebox(0,0)[lb]{\smash{{\SetFigFont{7}{8.4}{\rmdefault}{\mddefault}{\updefault}{\color[rgb]{0,0,0}$(1-p)^2p$}%
}}}}
\put(1276,-8491){\makebox(0,0)[lb]{\smash{{\SetFigFont{7}{8.4}{\rmdefault}{\mddefault}{\updefault}{\color[rgb]{0,0,0}$(1-p)^3$}%
}}}}
\put(2401,-8491){\makebox(0,0)[lb]{\smash{{\SetFigFont{7}{8.4}{\rmdefault}{\mddefault}{\updefault}{\color[rgb]{0,0,0}$(1-p)^2p$}%
}}}}
\put(5206,-8011){\makebox(0,0)[lb]{\smash{{\SetFigFont{7}{8.4}{\rmdefault}{\mddefault}{\updefault}{\color[rgb]{0,0,.82}$I_{LRL}$}%
}}}}
\put(4141,-8011){\makebox(0,0)[lb]{\smash{{\SetFigFont{7}{8.4}{\rmdefault}{\mddefault}{\updefault}{\color[rgb]{0,0,.82}$I_{RRL}$}%
}}}}
\put(2551,-8011){\makebox(0,0)[lb]{\smash{{\SetFigFont{7}{8.4}{\rmdefault}{\mddefault}{\updefault}{\color[rgb]{0,0,.82}$I_{RLL}$}%
}}}}
\put(11221,-8011){\makebox(0,0)[lb]{\smash{{\SetFigFont{7}{8.4}{\rmdefault}{\mddefault}{\updefault}{\color[rgb]{0,0,.82}$I_{LLR}$}%
}}}}
\put(6526,-6511){\makebox(0,0)[lb]{\smash{{\SetFigFont{12}{14.4}{\rmdefault}{\mddefault}{\updefault}{\color[rgb]{0,0,0}$2$}%
}}}}
\end{picture}%

%% file: gammas.pstex_t
\begin{picture}(0,0)%
\includegraphics{gammas.pstex}%
\end{picture}%
\setlength{\unitlength}{4144sp}%
\begingroup\makeatletter\ifx\SetFigFont\undefined%
\gdef\SetFigFont#1#2#3#4#5{%
  \reset@font\fontsize{#1}{#2pt}%
  \fontfamily{#3}\fontseries{#4}\fontshape{#5}%
  \selectfont}%
\fi\endgroup%
\begin{picture}(6137,2599)(-567,641)
\put(2042,1761){\makebox(0,0)[lb]{\smash{{\SetFigFont{8}{9.6}{\rmdefault}{\mddefault}{\updefault}{\color[rgb]{.69,0,.69}$k=3$}%
}}}}
\put(2042,2084){\makebox(0,0)[lb]{\smash{{\SetFigFont{8}{9.6}{\rmdefault}{\mddefault}{\updefault}{\color[rgb]{.69,0,.69}$k=4$}%
}}}}
\put(2042,2271){\makebox(0,0)[lb]{\smash{{\SetFigFont{8}{9.6}{\rmdefault}{\mddefault}{\updefault}{\color[rgb]{.69,0,.69}$k=5$}%
}}}}
\put(2042,2440){\makebox(0,0)[lb]{\smash{{\SetFigFont{8}{9.6}{\rmdefault}{\mddefault}{\updefault}{\color[rgb]{.69,0,.69}$k=10$}%
}}}}
\put(2046,2609){\makebox(0,0)[lb]{\smash{{\SetFigFont{10}{12.0}{\rmdefault}{\mddefault}{\updefault}{\color[rgb]{0,0,1}$\lambda=2$}%
}}}}
\put(-552,1206){\makebox(0,0)[lb]{\smash{{\SetFigFont{8}{9.6}{\rmdefault}{\mddefault}{\updefault}{\color[rgb]{0,.82,0}$\lambda=2.05$}%
}}}}
\put(-552,1441){\makebox(0,0)[lb]{\smash{{\SetFigFont{8}{9.6}{\rmdefault}{\mddefault}{\updefault}{\color[rgb]{0,.82,0}$\lambda=2.1$}%
}}}}
\put(-552,2187){\makebox(0,0)[lb]{\smash{{\SetFigFont{8}{9.6}{\rmdefault}{\mddefault}{\updefault}{\color[rgb]{0,.82,0}$\lambda=2.5$}%
}}}}
\put(-552,2728){\makebox(0,0)[lb]{\smash{{\SetFigFont{8}{9.6}{\rmdefault}{\mddefault}{\updefault}{\color[rgb]{0,.82,0}$\lambda=3$}%
}}}}
\put(5555,763){\makebox(0,0)[lb]{\smash{{\SetFigFont{12}{14.4}{\rmdefault}{\mddefault}{\updefault}{\color[rgb]{0,0,0}$p$}%
}}}}
\put(5195,656){\makebox(0,0)[lb]{\smash{{\SetFigFont{12}{14.4}{\rmdefault}{\mddefault}{\updefault}{\color[rgb]{0,0,0}1}%
}}}}
\put(3108,669){\makebox(0,0)[lb]{\smash{{\SetFigFont{12}{14.4}{\rmdefault}{\mddefault}{\updefault}{\color[rgb]{0,0,0}0}%
}}}}
\put(2304,763){\makebox(0,0)[lb]{\smash{{\SetFigFont{12}{14.4}{\rmdefault}{\mddefault}{\updefault}{\color[rgb]{0,0,0}$p$}%
}}}}
\put(1944,656){\makebox(0,0)[lb]{\smash{{\SetFigFont{12}{14.4}{\rmdefault}{\mddefault}{\updefault}{\color[rgb]{0,0,0}1}%
}}}}
\put(-143,669){\makebox(0,0)[lb]{\smash{{\SetFigFont{12}{14.4}{\rmdefault}{\mddefault}{\updefault}{\color[rgb]{0,0,0}0}%
}}}}
\put(-430,3044){\makebox(0,0)[lb]{\smash{{\SetFigFont{12}{14.4}{\rmdefault}{\mddefault}{\updefault}{\color[rgb]{0,0,0}$\gamma_{p,\lambda}$}%
}}}}
\put(2811,3044){\makebox(0,0)[lb]{\smash{{\SetFigFont{12}{14.4}{\rmdefault}{\mddefault}{\updefault}{\color[rgb]{0,0,0}$\widetilde \gamma_{p,\lambda}$}%
}}}}
\end{picture}%